\documentclass[10pt]{article}

\usepackage{amsmath,amssymb,latexsym}
\usepackage{amsthm}
\usepackage{array,graphicx}
\usepackage{hyperref,stmaryrd,xcolor} 

\newtheorem{theorem}{Theorem}[section]

\newtheorem{corollary}[theorem]{Corollary}  
\newtheorem{proposition}[theorem]{Proposition}
\newtheorem{lemma}[theorem]{Lemma}
\newtheorem{remark}[theorem]{Remark}
\newtheorem{assumption}{Assumption}

\numberwithin{equation}{section}

\def\la{\lambda}
\def\sig{\sigma}

\newcommand{\beq}{\begin{equation}}
\newcommand{\eeq}{\end{equation}}
\newcommand{\beqn}{\begin{eqnarray}}
\newcommand{\eeqn}{\end{eqnarray}}
\newcommand{\tn}{\textnormal}
\newcommand{\ds}{\displaystyle}
\newcommand{\norm}[1]{\left\Vert#1\right\Vert}
\newcommand\eh{\mathrm{e}_h}

\title{Convergence of an inverse problem \\for discrete wave equations\footnote{Partially supported by the Agence Nationale de la Recherche (ANR, France), Project C-QUID number BLAN-3-139579, Project CISIFS number NT09-437023 and the University Paul Sabatier (Toulouse 3), AO PICAN.}}
\author{Lucie Baudouin$^{1,2,}$\footnote{e-mail: {\tt baudouin@laas.fr}}\\
{\it\footnotesize $^{1}$ CNRS ; LAAS ; 7 avenue du colonel Roche, F-31077 Toulouse Cedex 4, France}\\
{\it\footnotesize $^{2}$ Universit\'e de Toulouse ; UPS, INSA, INP, ISAE, UT1, UTM ; LAAS ; F-31077 Toulouse, France.}\\
Sylvain Ervedoza$^{3,4,}$\footnote{e-mail: {\tt ervedoza@math.univ-toulouse.fr}}\\
{\it\footnotesize $^{3}$ CNRS ; Institut de Math\'ematiques de Toulouse UMR 5219 ; F-31062 Toulouse, France, }\\
{\it\footnotesize $^{4}$ Universit\'e de Toulouse ; UPS, INSA, INP, ISAE, UT1, UTM ; IMT ; F-31062 Toulouse, France.
}}

\begin{document}
 
\maketitle

\abstract{
It is by now well-known that one can recover a potential in the wave equation from the knowledge of the initial waves, the boundary data and the flux on a part of the boundary satisfying the Gamma-conditions of J.-L. Lions. We are interested in proving that trying to fit the discrete fluxes, given by discrete approximations of the wave equation, with the continuous one, one recovers, at the limit, the potential of the continuous model. In order to do that, we shall develop a Lax-type argument, usually used for convergence results of numerical schemes, which states that consistency and uniform stability imply convergence. In our case, the most difficult part of the analysis is the one corresponding to the uniform stability, that we shall prove using new uniform discrete Carleman estimates, where uniform means with respect to the discretization parameter. We shall then deduce a convergence result for the discrete inverse problems. Our analysis will be restricted to the $1$-d case for space semi-discrete wave equations discretized on a uniform mesh using a finite differences approach.
}\\



 \section{Introduction}
  
In this article, our goal is to study the convergence of an inverse problem for the $1$-d wave equation. Before introducing that problem, we shall present which inverse problem we are dealing with in the continuous setting.

\paragraph{The continuous inverse problem.}

For $T>0$, we consider the following continuous wave equation:
\begin{equation}
	\label{CWE1}
		\left\{\begin{array}{ll}
			\partial_{tt}y - \partial_{xx}y+ qy=g, \quad &(t,x) \in  (0,T)\times (0,1),
				\\
			y(t,0)=g^0(t),\quad y(t,1)=g^1(t),&  t\in (0,T),
				\\
			y(0,\cdot)= y^0, \quad \partial_t y(0,\cdot)= y^1. &
		\end{array}
		\right.
\end{equation}

Here, $y= y(t,x)$ is the amplitude of the waves, $(y^0, y^1)$ is the initial datum, $q = q(x)$ is a potential function, $g$ is a distributed source term and $(g^0,g^1)$ are boundary source terms.

Of course, this problem is well-posed in some functional spaces, for instance: If $(y^0, y^1) \in H^1(0,1) \times L^2(0,1)$, $g \in L^1(0,T;L^2(0,1))$, $g^i \in H^1(0,T)$ for $ i = 1, 2$, with the compatibility conditions $y^0(0) = g^0(0)$ and $y^0(1) = g^1(0)$ and $q \in L^\infty(0,1)$, the solution $y$ of \eqref{CWE1} belongs to $C([0,T]; H^1(0,1))\cap C^1([0,T], L^2(0,1))$.   
Such result is well-known except perhaps for the condition on the boundary data, which is a consequence of a hidden regularity result and a duality argument, giving a solution of \eqref{CWE1} in the sense of transposition - see \cite{Lions}, detailed for instance in \cite{LasieckaLionsTriggiani}.
Under this class of regularity, using again a hidden regularity result in \cite{Lions}, we can prove that $\partial_x  y(t,1)$ belongs to $L^2(0,T)$. 

We can therefore ask if, given $(y^0,y^1),\  g,\ (g^0,g^1)$, the knowledge of the additional information $\partial_x y(t,1)$ for a certain amount of time allows to characterize the potential $q$. We emphasize here that the data $(y^0,y^1),\ g, \ (g^0, g^1)$ are supposed to be known \emph{a priori}.

It has been proved in \cite{Baudouin01} that this question has a positive answer provided that $T$ is large enough ($T>1$ here) and $y \in H^1(0,T; L^\infty(0,1))$. Of course, to guarantee this regularity without any knowledge on $q$, we may impose some stronger conditions on the data $(y^0,y^1), \, g, \, (g^0,g^1)$, see e.g. in Remark \ref{RemReg} below.\\

Let us precisely recall the results in \cite{Baudouin01}. For $m \geq 0$, we introduce the set
$$
	L^\infty_{\leq m} (0,1) = \{q \in L^\infty(0,1), \ s.t. \, \norm{q}_{L^\infty(0,1)} \leq m \}.
$$
It will also be convenient to denote by $y[q]$ the solution $y$ of \eqref{CWE1} with potential $q$.
Assuming that $p\in L^\infty_{\leq m} (0,1)$ is a given potential, we are concerned with the stability of the map $ q \mapsto \partial_x y[q](\cdot,1)$  around $p$. Then we have the following local Lipschitz stability result: 
 
\begin{theorem}[\cite{Baudouin01}]\label{TCWE1}
Let $m>0$, $K>0$, $r>0$ and $T>1$.

Let $p$ in  $L^\infty_{\leq m}(0,1)$. Assume that the corresponding solution $y[p]$ of equation \eqref{CWE1} is such that 
\begin{equation}
	\label{RegularityCond}
	\norm{ y[p]}_{H^1(0,T;L^{\infty}(0,1))} \leq K. 
\end{equation}
Assume also that the initial datum $y^0$ satisfies
\begin{equation}
	\label{InitialDataCond}
	 \inf \left\{|y^0(x)|,x\in (0,1)\right\} \geq r.
\end{equation}
Then for all $q \in L^\infty_{\leq m}(0,1)$, $ \partial_{tx} y[p](\cdot,1)-\partial_{tx} y[q](\cdot,1) \in L^2(0,T)$ 
and there exists a constant $C>0$ that depends only on the parameters $(T,m, K,r)$ such that for all $q\in L^\infty_{\leq m}(0,1)$,
\begin{align}
	& 
	\norm{ \partial_{tx} y[p](\cdot,1)- \partial_{tx} y[q](\cdot,1) }_{L^2(0,T)} \leq  C \norm{p-q}_{L^2(0,1)},
	 \label{ReverseEst}
	\\
	& 
	\norm{ q-p}_{L^2(0,1)}
	\leq C \norm{ \partial_{tx} y[p](\cdot,1)-\partial_{tx} y[q](\cdot,1) }_{L^2(0,T)}.
	 \label{StabilityContinuous}
\end{align}
\end{theorem}

Estimate \eqref{StabilityContinuous} gives the Lispchitz stability of the inverse problem and \eqref{ReverseEst} states  the continuous dependance of the derivative of the flux of the solution with respect to the potential.
Together, these two estimates indicate that the above result is sharp. Note however that estimate \eqref{ReverseEst} is, by far, the easiest one to obtain.

\begin{remark}\label{RemReg}
	The condition \eqref{RegularityCond} can be guaranteed uniformly for $p \in L^\infty_{\leq m} (0,1)$ 
	with more constraints on the data $(y^0, y^1), \ g,\ (g^0,g^1)$ in \eqref{CWE1}, for instance:
	\begin{align*}
		& (y^0,y^1)  \in   H^2(0,1) \times H^1(0,1),
		\\
		& g  \in   W^{1,1}(0,T;L^2(0,1)),
		\quad
		(g^0, g^1) \in (H^2(0,T))^2,
	\end{align*}
	under the compatibility conditions
	$$
		 g^0 (0) = y^0(0), \quad g^1(0) = y^0(1), \quad \partial_t g^0(0) = y^1(0) \hbox{  and  } \partial_t g^1(0) = y^1(1).
	$$
	Indeed, under these assumptions, $\partial_t y[p]$ belongs to the space $C^0([0,T]; H^1(0,1))\cap C^1([0,T]; L^2(0,1))$ (see \cite{LasieckaLionsTriggiani}), with estimates 	depending only on $m$ and the norms of $(y^0, y^1),\, g, \, (g^0,g^1)$ in the above spaces. Therefore, due to Sobolev's imbedding, 
	$ y[p]$ satisfies \eqref{RegularityCond} for some constant $K>0$ that can be chosen uniformly with respect to  $p \in L^\infty_{\leq m} (0,1)$. 
\end{remark}

The method of proof of Theorem \ref{TCWE1} is based on a global Carleman estimate and is very close to the approach of \cite{ImYamCom01}, that concerns the wave equation with Neumann boundary condition and Dirichlet observation for the inverse problem of retrieving a potential. Actually, it also closely follows the approach of \cite{Yam99} but the work \cite{Baudouin01} requires less regularity conditions  on $y$. 

The use of Carleman estimates to prove uniqueness in inverse problems was introduced in \cite{BuKli81} by A. L. Bukhge{\u\i}m and M. V. Klibanov.  Concerning inverse problems for hyperbolic equations with a single observation, we can refer to  \cite{PuelYam96}, \cite{PuelYam97} or \cite{YamZhang03}, where the method relies on uniqueness results obtained by local Carleman estimates  (see e.g. \cite{Im02}, \cite{Isakov}) and compactness-uniqueness arguments based on observability inequalities (see also \cite{Zhang00}). Related references \cite{ImYamCom01}, \cite{ImYamIP01} and \cite{ImYamIP03}  use global Carleman estimates, but rather consider the case of interior or Dirichlet boundary data observation. Let us also mention the work \cite{BellassouedIP04} for logarithmic stability results when no geometric condition is fulfilled.


\paragraph{Discrete inverse problems.}

In this paper, we would like to address the question of the numerical computation of an approximation of the potential $p \in L^\infty(0,1)$, on which we assume the additional knowledge that its $L^\infty(0,1)$-norm is bounded by some constant $m>0$.

A natural approach is to find $p_h \in L^\infty_{\leq m}(0,1)$, or rather in a discrete version of it denoted by $L^\infty_{h,\leq m}(0,1)$ that will be made precise later (see \eqref{Linfini<m}), such that
\begin{equation}
	\label{ApproximationNaturelle}
	\partial_t \partial_x y_h[p_h]( t,1) \simeq \partial_t \partial_x y[p](t,1), \quad t \in (0,T),
\end{equation}
	where $y_h$ is the solution of a corresponding discrete wave equation with potential $p_h$ (here, $h>0$ refers to a discretization parameter) and the meaning of \eqref{ApproximationNaturelle} has to be clarified. The question is then the following: Does \eqref{ApproximationNaturelle} imply $p_h \simeq p$ ? Or, to be more precise, can we guarantee the convergence of the discrete potentials $p_h$ toward the continuous one $p$ ? 
	
Our analysis will focus on this precise convergence issue. To sum up in a very informal way our results, we will show that the convergence indeed holds true (Theorem~\ref{ThmCV}), provided a Tychonoff regularization process is introduced, and the key estimate is a stability estimate for the discrete inverse problem (Theorem~\ref{TCWE2}), given by appropriate global discrete Carleman estimates (Corollary~\ref{Corq} and Lemma \ref{LemDecompo}). \\

To be more precise, for $N \in \mathbb{N}$, set $h = 1/(N+1)$, and let us consider the following semi-discrete 1-d wave equation:
\begin{equation}
	\label{SDWE1}
	\left\{\begin{array}{ll}
		\partial_{tt}y_{j,h} - \left(\Delta_h y_{h}\right)_j+q_{j,h} y_{j,h}=g_{j,h}, \quad & t\in (0,T), \ j\in \llbracket 1,N \rrbracket,\\
		y_{0,h}(t)=g_h^0(t),\quad y_{N+1,h}(t)=g_h^1(t),& t\in [0,T],\\
		y_{j,h}(0)= y_{j,h}^0, \quad\partial_t y_{j,h}(0)= y_{j,h}^1, &  j\in \llbracket 1,N \rrbracket,
	\end{array}\right.
\end{equation}
where 
$$
(\Delta_h y_h)_j = \frac{1}{h^2}(y_{j+1,h} - 2 y_{j,h}+y_{j-1,h})
$$
denotes the classical finite-difference discretization of the Laplace operator and where $(y_{j,h}^0,y_{j,h}^1)$ are the initial sampled data at $x_j = jh$, $g_h^i\in L^2(0,T)$, $i=0,1$ and $ g_h \in L^1(0,T;L_h^2(0,1))$ are the boundary and source sampled data. Here and in the sequel, $L^2_h(0,1)$ denotes a discrete space endowed with a suitable discrete version of the $L^2(0,1)$-norm that will be introduced later (see \eqref{NormLp}). In the following, it will be important to sometimes underline the dependence of $y_h$ in \eqref{SDWE1} with respect to the potential $q_h$. This will be done using the notation $y_h[q_h]$.\\

Note that, using these notations, the discrete normal derivatives of $y_h[q_h]$ at the point $x = 1$ is naturally approximated by 
$$
	\frac{y_{N+1,h}[q_h] - y_{N,h}[q_h]}{h} = \frac{g_h^1(t)}{h} - \frac{y_{N,h}[q_h]}{h}.
$$
The fact that $g_h^1$ is known will allow us to simplify the difference of the discrete normal derivatives of $y_h[q_h]$ and $y_h[p_h]$ simply as
$$
	\frac{y_{N,h}[q_h]}{h} - \frac{y_{N,h}[p_h]}{h}.
$$

In order to prove the convergence of the inverse problem, we shall develop a Lax-type argument for the convergence of the numerical schemes that relies on:
\begin{itemize}
	\item {\bf Consistency:} If $p \in L^\infty_{\leq m } (0,1)$, there exists a sequence of discrete potentials $p_h \in L^\infty_{h,\leq m}(0,1)$ such that 
	\begin{align}
		\label{Consistency-1a}
		& p_h \underset{h \to 0}\longrightarrow p  \quad  \hbox{ in } L^2(0,1),
		\\
		\label{Consistency-1b}
		& \ds  \frac{\partial_t y_{N+1,h}[p_h] - \partial_t y_{N,h}[p_h]}{h} 
		\underset{h\to 0} \longrightarrow \partial_t\partial_x y[p](\cdot,1) \quad  \hbox{ in } L^2(0,T).
	\end{align}
	\item {\bf Uniform stability:} There exists a constant $C$ {\it independent of $h>0$} such that for all $(q_h, p_h) \in L^\infty_{h,\leq m} (0,1)^2$, 
		\begin{equation}
			\label{UniformStability}
			\norm{q_h-p_h}_{L_h^2(0,1)}\leq C \norm{ \frac{\partial_ty_{N,h}[q_h]}{h} - \frac{\partial_ty_{N,h}[p_h]}{h}}_{L^2(0,T)}.
		\end{equation} 
\end{itemize}
Of course, the consistency is the easiest part of the argument and will be detailed in Section~\ref{Cv}. The most difficult one comes from the stability estimate \eqref{UniformStability}. 

Actually, as we shall explain below, we will not get \eqref{UniformStability}, but we shall rather prove an estimate of the form
\begin{multline}
	\label{UniformStability-2}
	\norm{q_h-p_h}_{L_h^2(0,1)}\leq C \norm{ \frac{\partial_t y_{N,h}[q_h]}{h} - \frac{\partial_t y_{N,h}[p_h]}{h}}_{L^2(0,T)}\\
	 +C h \norm{\partial_h^+ \partial_{tt} y_h[q_h] -  \partial_h^+\partial_{tt} y_h[p_h]}_{L^2(0,T;L_h^2[0,1))},
\end{multline}
where 
$$
	(\partial_h^+ y_h)_j = \frac{y_{j+1,h} - y_{j,h}}{h},
$$
for some $C >0$ independent of $h >0$ - see Theorem \ref{TCWE2} for precise statements.

This is still compatible with the Lax argument: the added observation operator weakly converges to $0$ as $h \to 0$, since $h \partial_h^+$ is of norm bounded by $2$ on $L_h^2(0,1)$, and obviously converge to zero for smooth data. Therefore, in the limit $h \to 0$, this term disappears and \eqref{UniformStability-2} still yields~\eqref{StabilityContinuous}.

Of course, this should be taken into account into the consistency argument: Given $p \in L^\infty(0,1)$, one should find a sequence $p_h$ such that \eqref{Consistency-1a}--\eqref{Consistency-1b} hold and 
\begin{equation}
	\label{Consistency-2}
	h \partial_h^+\partial_{tt} y_h[p_h] \underset{h\to 0}{\longrightarrow} 0 \quad \hbox{in } L^2((0,T)\times (0,1)).
\end{equation}
We refer the reader to Theorem \ref{ThmConsistence} for precise assumptions and statements concerning the consistency.

The convergence result for the discrete inverse problems toward the continuous one is then given in Theorem \ref{ThmCV} and takes into account the previous comments. Roughly speaking, we will prove that, given any $p \in L^\infty(0,1)$ and any sequence $p_h \in L^\infty_{h, \leq m}(0,1)$ such that the convergences \eqref{Consistency-1b} and \eqref{Consistency-2} hold, the discrete potentials $p_h$ converge to $p$ in $L^2(0,1)$ as $h \to 0$. We refer to Section \ref{SubsecCv} for precise assumptions and statements.
\medskip

The proof of the uniform stability estimate \eqref{UniformStability-2} (see Theorem~\ref{TCWE2}) is based on a discrete Carleman estimate for \eqref{SDWE1}, which should be proved uniformly with respect to $h>0$ (see Corollary~\ref{Corq}). This is the main difficulty in our work. 

First, a discrete version of the continuous Carleman estimate yielding the stability \eqref{StabilityContinuous} cannot be true as it is. Indeed, that would contradict the results in \cite{GloLionsHe,InfZua,Zua05Survey,ErvZuaCime} that emphasize the lack of uniform observability of the discrete wave equations. This is due to the fact that the semidiscretization process that yields \eqref{SDWE1} creates spurious high-frequency solutions traveling at velocity of the order of $h$, see e.g. \cite{Tref,Mac2}. Hence, they cannot be observed in finite time uniformly with respect to $h>0$.

We shall therefore develop a discrete Carleman estimate for the discrete wave equation \eqref{SDWE1} which holds uniformly with respect to the discretization parameter $h>0$. We will use the same Carleman weights as in the continuous case. Though, the discrete integrations by parts will generate a term which cannot be handled directly. This will correspond to a term of the order of~$1$ at high-frequencies of the order of $1/h$, whereas it is small for frequencies of order less than $1/h$, thus being completely compatible with the continuous Carleman estimates and the analysis of the observability properties of the discrete wave equation. One can see \cite{Zua05Survey,ErvZuaCime} for review articles concerning that fact.

Uniform Carleman estimates for discrete equations have not been developed extensively so far. The only results we are aware of concern the elliptic case \cite{BoyerHubertLeRousseau,BoyerHubertLeRousseau2,BoyerHubertLeRousseau3} for applications to the controllability of discrete parabolic equations, in particular in \cite{BoyerHubertLeRousseau3}. More recently in \cite{ErvGournay}, discrete Carleman estimates have been derived for elliptic equations in order to prove uniform stability results for the discrete Calder\'on problems.

\paragraph{Outline.} The paper is organized as follows. Section~\ref{DCE} is devoted to the proof of discrete Carleman estimates for a 1-d semi-discrete wave operator. A uniform stability estimate for the related inverse problem is derived from it in  Section~\ref{InversePb}.  Convergence issues are finally detailed and proved in Section~\ref{Cv} and further comments are given in Section \ref{SecConclusion}.

\section{Discrete Carleman estimates}\label{DCE}

In this section, we establish uniform Carleman estimates for the semi-discrete wave operator. 

\subsection{Continuous case}

We recall here the global Carleman estimates for the continuous wave operator. 

Let $x^0 <0$, $s>0$, $\la>0$ and $\beta \in (0,1)$.  On $[-T,T]\times [0,1]$, we define the weight functions $\psi =\psi(t,x)$ and $\varphi =\varphi(t,x)$ 
as
\begin{eqnarray}
	\psi (t,x)= | x-x^0|^2 -\beta t^2+C_0,
	& & ~\varphi(t,x) = e^{\la \psi(t,x)},
	\label{varphi}
\end{eqnarray}
where $C_0>0$ is such that $\psi\geq 1$ on $[-T,T]\times[0,1].$

Let us begin by recalling the continuous Carleman estimate. This will make easier the comparisons with the forthcoming discrete ones:
 
\begin{theorem}[\cite{Baudouin01}]\label{CarlemanC}
Let $L w=\partial_{tt}w - \partial_{xx}w$, $T>0$ and $\beta \in (0,1)$.

There exist $\lambda_0>0$, $s_0>0$ and a constant $M=M(s_{0},\la_0,T,\beta,x^0) >0$ such that for all $s\geq s_{0}$, $\lambda \geq \lambda_0$ and $w$ satisfying
$$
\left\lbrace
	\begin{array}{lll}
		Lw \in L^2((-T,T) \times (0,1)),\\
		w \in L^2(-T,T;H^1_0(0,1)),\\
		w(\pm T,\cdot)=\partial_t w( \pm  T,\cdot)=0,
	\end{array}
\right.
$$
we have
\begin{multline}
	s\la \int_{-T}^{T} \int_0^1 \varphi e^{2s\varphi} \left (| \partial_t w|^2 + |\partial_x w |^2\right)\,dxdt 
	+ s^3\la^3\int_{-T}^{T} \int_0^1 \varphi^3 e^{2s\varphi}| w|^2\,dxdt 
		\\
	\leq M\int_{-T}^{T} \int_0^1 e^{2s\varphi}| Lw|^2\,dxdt 
	+ Ms\la \int_{-T}^{T} \varphi(t,1) e^{2s\varphi(t,1)}\left|\partial_x w(t,1)\right|^2\,dt.
	\label{CarlemC}
\end{multline}
\end{theorem}

Carleman estimates for hyperbolic equations can be found in \cite{Im02} and we refer the reader to the bibliography therein for extensive references. The Carleman estimate stated here  can be seen as a more refined version of the one in \cite[Theorem 1.2]{Im02} but in the case of boundary observation and with the freedom on $\lambda$. For the proof of this Carleman estimate, we therefore refer to \cite{Baudouin01}.


\begin{remark}
	Note that the above Carleman estimate holds without any condition on $T$. This might be surprising but this should not be since we assume that $w (\pm T) = \partial_t w (\pm T) = 0$ and therefore, the corresponding unique continuation result is: If $w (\pm T) = \partial_t w (\pm T) = 0$, $w \in L^2(-T,T;H^1_0(0,1))$, $\partial_{tt} w - \partial_{xx} w = 0$ and $\partial_x w(\cdot,1) = 0$, then $w \equiv 0$.
\end{remark}

\subsection{Statement of the results}

In this section, we state uniform Carleman estimates for semi-discrete wave operators.

Of course, since we work in a semi-discrete framework, the space variable $x$ is now to be considered as taking only discrete values $x_j = jh\in [0,1]$ for $j\in\llbracket 0, \dots, N+1\rrbracket $ (recall that $h = 1/(N+1)$). Therefore, for continuous functions $f$ (e.g. with $\varphi$, $\psi$,...), we will write indifferently $f(x_j)$ or $f_j$. We will also add the subscript $h$ when we want to emphasize the dependence in the mesh size parameter $h>0$, but we shall remove it as soon as the context clearly underlines that we are working for one particular $h>0$.

Moreover, by analogy with the continuous case, we will use the following notations:
\begin{equation}\label{intf}
	\ds\int_{(0,1)} f_h = h \ds\sum_{j=1}^{N} f_{j,h} , 
	\quad
	 \ds\int_{[0,1)} f_h 	= h \ds\sum_{j=0}^{N} f_{j,h}, 
	 \quad 
	 \ds\int_{(0,1]} f_h 	= h \ds\sum_{j=1}^{N+1} f_{j,h},  .
\end{equation}
Note that it also defines in a natural way a discrete version of the $L^p(0,1)$-norms as follows: for $p\in[1, \infty)$, we introduce $L^p_h(0,1)$ (respectively $L^p_h([0,1))$) the space of discrete functions $f_h$ defined for $jh$, $j \in \llbracket 1, N\rrbracket$, (resp. $j \in \llbracket 0, N\rrbracket$) endowed with the norms
\begin{equation}
	\label{NormLp}
	\norm{f_h}_{L^p_h(0,1)}^p = 	\ds\int_{(0,1)} |f_h|^p  
	\quad \hbox{(resp. } \norm{f_h}_{L^p_h([0,1))}^p = 	\ds\int_{[0,1)} |f_h|^p \,\hbox{)},
\end{equation}
and, for $p = \infty$, 
\begin{equation}
	\label{NormLinfini}
	\norm{f_h}_{L^\infty_h(0,1)} = 	\sup_{ j \in \llbracket 1, N\rrbracket} |f_{j,h}| 
	\quad \hbox{(resp. } \norm{f_h}_{L^\infty_h([0,1))} = 	\sup_{j \in \llbracket 0, N\rrbracket} |f_{j,h}| \hbox{)}.
\end{equation}
By analogy with $L^\infty_{\leq m}(0,1)$, we also define
\begin{equation}
	\label{Linfini<m}
	L^\infty_{h,\leq m} (0,1)  = \left\{q_h = (q_{j,h})_{j \in \llbracket 1, N \rrbracket} \in L_h^\infty(0,1), 
	 \ s.t. \,\norm{q_{h}}_{L^\infty_h(0,1)} \leq m  \right\}.
\end{equation}
We shall also use in the sequel the following notations:
\begin{eqnarray*}
	(m_h v_h)_j =  \dfrac{v_{j+1,h} + 2 v_{j,h}+v_{j-1,h}}{4} ~; &
	&
	 (m^+_h v_h)_j  = (m_h^- v_h )_{j+1} = \dfrac{v_{j+1,h} + v_{j,h}}{2}   ~;
	\\
	(\partial_h v_h)_{j} = \dfrac{v_{j+1,h} -v_{j-1,h}}{2h}~;& 
	&
	(\partial^+_h v_h)_{j} = (\partial^-_h v_h)_{j+1} = \dfrac{v_{j+1,h} -v_{j,h}}{h}  ~;
	\\
	(\Delta_h v_h)_j & = &\dfrac{v_{j+1,h} - 2 v_j+v_{j-1,h}}{h^2}.
\end{eqnarray*}

One of the main results of this paper is the following discrete Carleman estimates:
\begin{theorem}\label{CarlemanD}
Let $L_h w_h=\partial_{tt}w_h - \Delta_h w_h$, $T>0$ and $\beta \in (0,1)$.
\\
There exist $s_0>0$, $\lambda>0$, $\varepsilon >0$, $h_0>0$ and a constant $M=M(s_0,\la, T, \varepsilon,\beta) >0$ independent of $h>0$ such that for all $ h\in (0,h_0)$ and $ s\in \left(s_{0},\varepsilon/ h\right)$, for all $w_h$ satisfying
$$
	\left\lbrace
		\begin{array}{l}
			L_h w_h \in L^2(-T,T; L^2_h(0,1)),
			\\
			w_{0,h}(t) = w_{N+1,h}(t) = 0 \quad \textit{ on }  (-T,T),
			\\	
			w_h(\pm T) =  \partial_t w_h (\pm T) =   0,
		\end{array}
	\right.
$$
we have 
\begin{gather}
	s \int_{-T}^{T} \int_{(0,1)} e^{2s\varphi}| \partial_t w_h|^2 \, dt
	+
	s \int_{-T}^{T} \int_{[0,1)} e^{2s\varphi} |\partial_h^+ w_h |^2\, dt
	+ s^3\int_{-T}^{T} \int_{(0,1)} e^{2s\varphi}| w_h |^2 \, dt
	\nonumber \\
	\leq M\int_{-T}^{T} \int_{(0,1)} e^{2s\varphi}| L_h w_h|^2 \, dt  + 
	Ms \int_{-T}^{T} e^{2s\varphi(t,1)} \left|(\partial^-_h w_h)_{N+1} \right|^2 \, dt
	\label{CarlemD}\\
	+ Ms \int_{-T}^{T} \int_{[0,1)} e^{2s\varphi} |h\partial_h^+\partial_t w_h |^2\, dt  \nonumber
\end{gather}
for $\varphi$ given by \eqref{varphi}.
\end{theorem}
The proof of Theorem \ref{CarlemanD} will be given at the end of Section~\ref{CARL}. \\
The following remarks are in order:

$\bullet$
	The weight function $\varphi$ in the above discrete Carleman estimate is the same as for the continuous one. 

$\bullet$
	In Theorem \ref{CarlemanD}, the parameter $\lambda$ is fixed, whereas it is not in the continuous Carleman estimate of Theorem \ref{CarlemanC}. Looking carefully at the proof of Theorem \ref{CarlemanD}, one can prove that there exists $\lambda_0$ such that, for all $\lambda \geq \lambda_0$, there exist $\varepsilon(\lambda)>0$ and $M = M(\lambda)$ such that \eqref{CarlemD} holds for all $s \geq s_0(\lambda)$ and $sh \leq \varepsilon(\lambda)$. These dependences of $\varepsilon$ and $M$ on $\lambda$ are very intricate and we did not manage to follow it precisely.

$\bullet$ The fact that $M$ is independent of $h>0$ is of major importance in the applications we have in mind. This is very similar to the observability properties of discrete wave equations  for which one should prove observability results uniformly with respect to the discretization parameter(s), otherwise the discrete controls (obtained by duality from the discrete observability properties) may diverge, see e.g. \cite{ErvZuaCime}.

$\bullet$
	The range of $s$ in Theorem \ref{CarlemanD} is limited to $s \leq \varepsilon/h$. This is a technical assumption, that is not surprising when comparing it to \cite{BoyerHubertLeRousseau,BoyerHubertLeRousseau2}. Indeed, for $s$ of the order of $1/h$, $e^{s \varphi}$ is a high-frequency function of frequency of the order of $1/h$ and therefore it does not reflect anymore  the dynamics of the continuous wave operator.

$\bullet$
	A new term appears in the right hand side of \eqref{CarlemD}, which cannot be absorbed by the terms of the left hand side. Though, this term is needed and cannot be removed. Otherwise, one could obtain a uniform observability result for the discrete wave equation, a fact which is well-known to be false according to \cite{InfZua}. Besides, this extra term is of the order of one for frequencies of the order of $1/h$, whereas it can be absorbed by the left hand side for frequencies of smaller order. According to \cite{ErvZuaCime}, this indicates that the extra term in estimate \eqref{CarlemD} has the right scale.\\

Note that in the application we have in mind, we shall not use directly the Carleman estimate \eqref{CarlemD} which involves the wave equation without a potential but rather one in which a $L^\infty$ potential is allowed.
Indeed, we have the following corollary:
\begin{corollary}\label{Corq}
Let $T>0$ and $\beta \in (0,1)$. 
Let $m >0$, $q_h \in L^\infty_{h,\leq m} (0,1)$ and $~L_h[q_h] w_h=\partial_{tt}w_h - \Delta_h w_h + q_h w_h$.
\\
There exist $s_0>0$, $\lambda>0$, $\varepsilon >0$, $h_0>0$ and a constant $M=M(s_0,\la, T,m, \varepsilon,\beta) >0$ such that for all $ h\in (0,h_0)$ and for all $ s\in \left(s_{0}, \varepsilon/ h\right)$, for all $w_h$ satisfying
$$
	\left\lbrace
		\begin{array}{l}
		L_h[q_h] w_h \in L^2(-T,T; L^2_h(0,1)),\\
		w_{0,h}(t) = w_{N+1,h}(t) = 0 \quad \textit{ on } (-T,T),\\
		w_h(\pm T) =  \partial_t w_h (\pm T) =   0,
	\end{array}
	\right.
$$
we have:
\begin{gather}
	s \int_{-T}^{T} \int_{(0,1)} e^{2s\varphi}| \partial_t w_h|^2 \, dt
	+s \int_{-T}^{T} \int_{[0,1)} e^{2s\varphi} |\partial_h^+ w_h |^2 \, dt
	+ s^3\int_{-T}^{T} \int_{(0,1)} e^{2s\varphi}| w_h|^2   \, dt
 \nonumber \\
	\leq
	 M\int_{-T}^{T} \int_{(0,1)} e^{2s\varphi}| L_h[q_h] w_h |^2 \, dt
	 + 	
	Ms \int_{-T}^{T} e^{2s\varphi(t,1)} \left|(\partial^-_h w_h)_{N+1} \right|^2 \, dt
	\label{CarlemDq}\\
	+ Ms \int_{-T}^{T} \int_{[0,1)} e^{2s\varphi} |h\partial_h^+\partial_t w_h |^2 \, dt, \nonumber
\end{gather}
for $\varphi$ given by \eqref{varphi}.
\end{corollary}

\begin{proof}
This is a simple consequence of Theorem \ref{CarlemanD}, since $L_h w_h = L_h[q_h] w_h - q_h w_h$ 
with $q_h \in L^\infty_{h,\leq m} (0,1)$ leads to
$$
	\int_{-T}^{T} \int_{(0,1)} e^{2s\varphi}| L_h w_h |^2 \, dt  \leq 2 	\int_{-T}^{T} \int_{(0,1)} e^{2s\varphi}| L_h[q_h] w_h |^2 \, dt 
	+ 2 m^2 \int_{-T}^{T} \int_{(0,1)} e^{2s\varphi}|w_h |^2 \, dt .
$$
This last term can be absorb by the left hand side of \eqref{CarlemD} by choosing $s$ large enough. This immediately yields \eqref{CarlemDq}.
\end{proof}

Until the end of this section, we shall work for $h>0$ fix. We therefore omit the indexes $h$ on the discrete functions to simplify notations. 

\subsection{Basic discrete identities}

Below, we list several preliminary dentities that will be extensively used in the sequel.
Let us begin with easy identities left to the reader:
\begin{lemma} \label{ID}
The following identities hold:
\begin{eqnarray}
\dfrac{a_1b_1+a_2b_2}2 \hspace{-2ex}& =&  \hspace{-2ex}\left(\dfrac{a_1+a_2}2\right)  \left(\dfrac{b_1+b_2}2\right) + \dfrac{h^2}4\left(\dfrac{a_1-a_2}h\right)  \left(\dfrac{b_1-b_2}h\right) ;
\label{utile+}
	\\
\dfrac{a_1b_1-a_2b_2}h \hspace{-2ex}&= & \hspace{-2ex} \left(\dfrac{a_1-a_2}h\right)   \left(\dfrac{b_1+b_2}2\right)  +  \left(\dfrac{a_1+a_2}2\right)   \left(\dfrac{b_1-b_2}h\right).\label{utile-}
\end{eqnarray}
\end{lemma}

Using these identities, one can obtain the next lemma:
\begin{lemma} \label{IPP}
The following identities hold:
\begin{align}
	m^+_h  &= I + \frac h2 \partial_h^+ ; 
	\quad m_h  = I+\frac{h^2}4 \Delta_h = m^+_h m^-_h~;
	\label{mhv}
	\\
	\partial_h  &= \dfrac 12 (\partial^+_h  + \partial^-_h )
	 = m^+_h \partial_h^-  =\partial_h^- m^+_h  =  m^-_h \partial_h^+~ = \partial_h^+ m^-_h;
	\label{dhv}
	\\
	\Delta_h  &=\partial^+_h\partial^-_h = \partial^-_h\partial^+_h~;
	\label{Dhv}
	\\
	m^+_h (uv) &= (m^+_h u)(m^+_h v) +\frac{h^2}4 (\partial ^+_h u) (\partial^+_h v)~;
	\label{mhuv}
	\\
	\partial^\pm_h (uv) &= (\partial^\pm_h u)(m^\pm_h v) + (m^\pm_h u) (\partial^\pm_h v)~;
	\label{dhuv}
	\\
	\Delta_h (\rho v) &= (\Delta_h \rho) \, (m_h v)  + 2(\partial_h \rho)\, ( \partial_h v) + (m_h \rho)\, (\Delta_h v).
	\label{Dhrv}
\end{align}
\end{lemma}

\begin{proof}
To begin with, one easily obtains \eqref{mhv}, since
$$
	\left\{\begin{array}{l}
		\ds (m^+_h v)_j =  \dfrac{v_{j+1} + v_{j}}{2} = v_j + \frac h2  \dfrac{v_{j+1} -v_{j}}{h} = v_j + \frac h2 (\partial_h^+v)_j,
		\\
		\ds \left(m_hv\right)_j =  \dfrac{v_{j+1} + 2 v_ j+ v_{j-1}}{4}  = v_j +  \dfrac{v_{j+1} - 2 v_ j+ v_{j-1}}{4 } = v_j + \dfrac{h^2}4 \left(\Delta_h v \right)_j.
	\end{array}\right.
$$
Similar computations left to the readers yield \eqref{dhv} and \eqref{Dhv}.

Identities \eqref{mhuv}--\eqref{dhuv} are straightforward consequences of the formula of Lemma \ref{ID}. To get \eqref{Dhrv}, we do as follows:
\begin{align*}
	\Delta_h (\rho v) &= \partial^-_h\left( \partial^+_h (\rho v)\right)\\
	&=\partial^-_h\left( (\partial^+_h \rho)(m^+_h v) + (m^+_h \rho) (\partial^+_h v) \right)\\
	&=(\partial^-_h \partial^+_h \rho)(m^-_hm^+_h v)  + (m^-_h \partial^+_h \rho)(\partial^-_h m^+_h v) 
		+ (\partial^-_hm^+_h \rho) (m^-_h\partial^+_h v) + (m^-_hm^+_h \rho) (\partial^-_h\partial^+_h v)\\
	&=(\Delta_h \rho)(m_h v)  + 2 ( \partial_h \rho)(\partial_h v)  + (m_h\rho) (\Delta_h v).
	\end{align*}
Note that this should of course be compared to the corresponding classical Leibniz formula $\Delta(\rho v) = v\Delta\rho + 2 \nabla\rho\cdot\nabla v + \rho \Delta v$.
\end{proof}

We now explain how discrete integrations by parts work:
\begin{lemma}[Discrete integration by parts formula]
	\label{LemIPP2}
	Let $v,f,g$ be discrete functions such that $v_0=v_{N+1}=0$. Then we have the following identities:
	\begin{align}
		&\bullet~ \int_{[0,1)} g (\partial_h^+ f )=  - \int_{(0,1]} (\partial_h^- g) f  + g_{N+1} f_{N+1}  - g_0 f_0 ~ ;
	\label{IPP00} 
		\\
		&\bullet~  \int_{(0,1)} g (\partial_h f )=  \int_{[0,1)} (m_h^+ g)(\partial_h^+ f)
		 - \frac{h}{2} g_0 (\partial_h^+ f)_0- \frac{h}{2} g_{N+1} (\partial_h^- f)_{N+1} ~ ;
	\label{IPP000} 
		\\
		&\bullet~   2 \int_{(0,1)} g v(\partial_h v)  = - \int_{(0,1)}|v|^2 ~\partial_h g 
		+ \dfrac{h^2}2 \int_{[0,1)} |\partial_h^+ v|^2 \partial_h^+g ~ ;
	\label{IPP0}\\
		 &\bullet~ \int_{(0,1)} g(\Delta_h v)  = - \int_{[0,1)}(\partial^+_h v) ~(\partial^+_h g)
		 - (\partial_h^+ v)_0 g_0+  (\partial_h^- v)_{N+1} g_{N+1} ~ ;
	\label{IPP1}\\
		 &\bullet~ \int_{(0,1)} g v (\Delta_h v)  = - \int_{[0,1)}(\partial^+_h v)^2 ~(m^+_h g) + \frac{1}{2} \int_{(0,1)} |v|^2 \Delta_h g ~ ;
		 \label{IPP1bis}\\
		 &\bullet~ \int_{(0,1)} g\Delta_h v \partial_h v =- \dfrac 12\int_{[0,1)} |\partial_h^+ v|^2 \partial_h^+g 
		 + \dfrac 12 \left|(\partial_h^- v)_{N+1} \right|^2 g_{N+1} - \dfrac 12\left|(\partial_h^+ v)_0 \right|^2 g_{0}.
		\label{IPP2}
	\end{align}
\end{lemma}

\begin{proof}
Let us begin with \eqref{IPP00}:
\begin{eqnarray*}
	\int_{[0,1)} g \partial_h^+ f 
	&= &
	h \sum_{j=0}^N g_j \left(\frac{f_{j+1}- f_j}{h}\right)
	= \sum_{j= 0}^N g_j f_{j+1} - \sum_{j=0}^N g_j f_j
	\\
	& = &
	 \sum_{j=1}^{N+1} g_{j-1} f_j -  \sum_{j=1}^{N+1} g_j f_j  + g_{N+1} f_{N+1} - g_0 f_0
	 \\
	 & = &
	 - h \sum_{j=1}^{N+1}\left( \frac{g_{j}- g_{j-1}}{h}\right) f_j  + g_{N+1} f_{N+1}- g_0 f_0.
\end{eqnarray*}
In order to prove \eqref{IPP000}, using \eqref{dhv}, we do as follows:
\begin{eqnarray*}
	\int_{(0,1)} g \partial_h f 
	& = &
	\frac{1}{2} \left(\int_{(0,1)} g \partial_h^- f + \int_{(0,1)} g \partial_h^+ f\right)
	\\
	& = &
	\frac{h}{2} \sum_{j=1}^N g_j (\partial_h^+ f)_{j-1} +	\frac{h}{2} \sum_{j=1}^N g_j (\partial_h^+ f)_{j}
	\\
	&= &
	\frac{h}{2} \sum_{j=0}^{N-1} g_{j+1} (\partial_h^+ f)_{j} +	\frac{h}{2} \sum_{j=1}^N g_j (\partial_h^+ f)_{j}  
	\\
	&=&
	\frac{h}{2} \sum_{j=0}^{N}(g_j+ g_{j+1}) (\partial_h^+ f)_{j} 
	- \frac{h}{2} g_0 (\partial_h^+ f)_0- \frac{h}{2} g_{N+1} (\partial_h^+ f)_N.
\end{eqnarray*}
To prove \eqref{IPP0}, using the fact that $v_0 =v_{N+1}=0$, and successively \eqref{IPP000}, \eqref{mhuv},  \eqref{dhuv} and \eqref{IPP00}, we obtain:
\begin{eqnarray*}
	2 \int_{(0,1)} g v\partial_h v
	& = &
	2  \int_{[0,1)} m_h^+ (v g) (\partial_h^+ v)
	\\
	& = &
	2 \int_{[0,1)}\left((m_h^+ v) (m_h^+ g) + \frac{h^2}{4}( \partial_h^+ v)  (\partial_h^+ g) \right)  (\partial_h^+ v) 
	\\
	& = & \int_{[0,1)} (m_h^+ g) \partial_h^+( |v|^2) + \frac{h^2}{2} \int_{[0,1)} ( \partial_h^+ v)^2  (\partial_h^+ g) 
	\\
	& = & - \int_{(0,1]} (\partial_h^- (m_h^+ g)) |v|^2+ \frac{h^2}{2} \int_{[0,1)} ( \partial_h^+ v)^2  (\partial_h^+ g) 
	\\
	& = & - \int_{(0,1)} (\partial_h g) |v|^2+ \frac{h^2}{2} \int_{[0,1)} ( \partial_h^+ v)^2  (\partial_h^+ g). 
\end{eqnarray*}
For \eqref{IPP1}, using \eqref{IPP00}, we write
\begin{eqnarray*}
	\int_{(0,1)} g (\Delta_h v )
	& =  &
	\int_{(0,1)} g \, \partial^-_h \partial_h^+ v  
	\\
	&=&  \int_{(0,1]} g \, \partial^-_h \partial_h^+ v -  g_{N+1} ((\partial_h^+ v)_{N+1} - (\partial_h^- v)_{N+1} )
	\\
	& = & 
	- \int_{[0,1)}(\partial^+_h v )~(\partial^+_h g)- (\partial_h^+ v)_0 g_0+  (\partial_h^- v)_{N+1} g_{N+1}.
\end{eqnarray*}
From \eqref{IPP1}, we prove \eqref{IPP1bis}, using $v_0 =v_{N+1}=0$ and Lemma \ref{IPP}:
\begin{eqnarray*}
	\int_{(0,1)}g v \Delta_h v
	& =  & - \int_{[0,1)} (\partial_h^+ v) (\partial_h^+ (gv)) 
	\\
	& =& -\int_{[0,1)} |\partial_h^+ v|^2 m_h^+ g - \int_{[0,1)} \partial_h^+ v \, m_h^+ v \, \partial_h^+ g
	\\
	& = &-\int_{[0,1)} |\partial_h^+ v|^2 m_h^+ g - \frac{1}{2}\int_{[0,1)} \partial_h^+ (|v|^2) \partial_h^+ g
	\\
	& = &-\int_{[0,1)} |\partial_h^+ v|^2 m_h^+ g + \frac{1}{2}\int_{(0,1)} |v|^2 \Delta_h g.
\end{eqnarray*}

Finally, in order to prove \eqref{IPP2}, we first remark that, using Lemma \ref{IPP}, 
$$
	(\Delta_h v)_j (\partial_h v)_j= (\partial_h^- (\partial_h^+ v))_j (m_h^- (\partial_h^+ v))_j = \frac{1}{2}( \partial_h^- \left(|\partial_h^+ v|^2 \right))_j 
$$
and therefore \eqref{IPP2} follows from \eqref{IPP00}:
\begin{eqnarray*}
	\lefteqn{ 
	\int_{(0,1)} g\Delta_hv \partial_h v  = \dfrac 12 	\int_{(0,1)} g\partial_h^- \left(|\partial_h^+ v|^2 \right)
	}\\
	& = &
	\dfrac 12 	\int_{(0,1]} g\partial_h^- \left(|\partial_h^+ v|^2 \right) - \frac{1}{2} g_{N+1} \left(|(\partial_h^+ v)_{N+1}|^2 - |(\partial_h^- v)_{N+1}|^2 \right) 
	\\
	& = & - \dfrac 12\int_{[0,1)} |\partial_h^+ v|^2 \partial_h^+g 
		 + \dfrac 12 \left|(\partial_h^- v)_{N+1} \right|^2 g_{N+1} - \dfrac 12\left|(\partial_h^+ v)_0 \right|^2 g_{0}.
\end{eqnarray*}
This concludes the proof of Lemma \ref{LemIPP2}.
\end{proof}

\subsection{Computation of the conjugate operator}\label{DEC3}

Set  $\rho = \exp(- s\varphi)$, $\varphi$ given by \eqref{varphi}, and 
\begin{equation}
	\label{PhDef}
	v(t,x) = \rho^{-1} (t,x)w(t,x) = e^{s\varphi(t,x)}w(t,x) \quad \hbox{ and } \quad  P_h v:= \frac{1}{\rho} \left( \partial_{tt} - \Delta_h \right) (\rho v).
\end{equation}

\begin{proposition}\label{B}
The conjugate operator $P_h$ can be expanded as follows:
\begin{equation}
	P_h v = \partial_{tt}v + 2\partial_t v \dfrac{\partial_t\rho}{ \rho}+  v \dfrac{\partial_{tt}\rho}{ \rho} 
		\label{ConjugateOperator}
	- \left(1 + \dfrac{h^2}2 \dfrac{\Delta_h \rho}{\rho} \right)\Delta_h v  -2 \partial_h v \dfrac{\partial_h\rho}{ \rho} - v \dfrac{\Delta_h\rho}{ \rho}.
\end{equation}
\end{proposition}

\begin{proof}
Identity \eqref{ConjugateOperator} can be deduced easily by explicit computations based on Lemma \ref{ID}. 
Indeed, 
$$
	P_h v 
	= 
	\dfrac 1\rho[\partial_{tt} - \Delta_h ](\rho v)
	=
	\partial_{tt} v + 2\partial_t v\dfrac {\partial_t\rho}\rho +v \dfrac {\partial_{tt}\rho}\rho - \dfrac {\Delta_h(\rho v)}\rho.
$$
But, using \eqref{Dhrv}, we get
$$
\dfrac {\Delta_h(\rho v)}\rho =  \Delta_h v \dfrac{m_h\rho}{\rho} +2 \partial_h v \dfrac{\partial_h\rho}{ \rho} + m_h v \dfrac{\Delta_h\rho}{ \rho}.
$$
Besides, from \eqref{mhv},
$$
	\left(\dfrac{m_h\rho}{\rho} \right)=  1 + \dfrac{h^2}4 \left(\dfrac{\Delta_h \rho}{\rho} \right)\hbox{  and } m_h v = v + \frac{h^2}{4} \Delta_h v,
$$
which immediately yield identity \eqref{ConjugateOperator}.
\end{proof}

One step of the usual way to prove a Carleman estimate is to split the operator $P_h$ into two operators $P_{h,1}$ and $P_{h,2}$ (detailed in Section \ref{CARL}), that, roughly speaking, corresponds to a decomposition into a self-adjoint part and a skew-adjoint one, and then to compute and estimate the scalar product 
$$
\int_{-T}^{T} \int_{(0,1)} P_{h,1}v\, P_{h,2}v\, dt.
$$
But we first need  to give a more precise expression of $P_h v$, using the following equalities:

\begin{proposition} \label{A}
The coefficients in the expression of $P_h$ can be expanded as follows:
\begin{eqnarray}
	\label{dtr}
	\dfrac{\partial_t\rho}{ \rho}  =  -s\la \varphi \partial_t \psi,
	&\quad &
	\dfrac{\partial_{tt}\rho}{ \rho}  =  s^2\la^2\varphi^2 \left( \partial_t \psi \right)^2  -s\la^2\varphi\left( \partial_t \psi \right)^2  -s\la\varphi \partial_{tt} \psi, 
	\\
	\dfrac{\partial_h\rho}{ \rho} =  -s\la A_1,
	\label{dhr}
	&\quad  & 
	\dfrac{\Delta_h \rho}{\rho}  =  s^2\la^2  A_2 - s\la^2 A_3 - s\la A_4,
\end{eqnarray}
where the coefficients $(A_1, A_2, A_3, A_4)$ are given by
\begin{align}
	A_{1} (t,x)
	&=
	  \dfrac 12\int_{-1}^1 \left[\varphi \partial_x\psi\right](t, x+\sig h) \dfrac{e^{-s\varphi(t, x+\sig h)}}{e^{-s\varphi(t,x)}}\,d\sig, 
	  \label{A1}\\
	A_{2}(t,x)
	&= \int_{-1}^1 (1-|\sig|) \left[ \varphi^2 (\partial_x\psi)^2\right](t,x+\sig h) \dfrac{e^{-s\varphi(t,x_j+\sig h)}}{e^{-s\varphi(t,x)}}\,d\sig,
	\label{A2}\\
	A_{3}(t,x) &= \int_{-1}^1 (1-|\sig|) \left[\varphi (\partial_x\psi)^2 \right](t,x+\sig h) \dfrac{e^{-s\varphi(t,x_j+\sig h)}}{e^{-s\varphi(t,x)}}\,d\sig,
	\label{A3}\\
	A_{4}(t,x) &= \int_{-1}^1 (1-|\sig|) \left[ \varphi\partial_{xx}\psi\right](t,x+\sig h) \dfrac{e^{-s\varphi(t,x_j+\sig h)}}{e^{-s\varphi(t,x)}}\,d\sig.\label{A4}
\end{align}
\end{proposition}

\begin{proof}
Since $\rho = e^{-s\varphi} $ and  $\varphi = e^{\la \psi}$, identities \eqref{dtr} are straightforward.

Getting \eqref{dhr} is more technical. We write
$$
\left(\partial_h\rho \right)_j= \dfrac{\rho_{j+1} -\rho_{j-1}}{2h}
= \dfrac 1{2h}\int_{x_j-h}^{x_j+h} \partial_x \rho(x)\,dx = \dfrac 12 \int_{-1}^{1} \partial_x \rho(x_j+ \sig h)\,d \sig
$$
and since $\partial_x \rho = -s\la\rho\varphi\partial_x \psi$, we get \eqref{dhr}$_1$:
$$
\left(\dfrac{\partial_h\rho}{ \rho}\right)_j(t) = - \dfrac {s\la}2 \int_{-1}^{1} [\varphi\partial_x \psi](t,x_j+ \sig h) \dfrac {\rho(t,x_j+ \sig h)}{\rho(t,x_j)}\,d \sig = - s\la A_{1,j}(t).
$$

Similarly, the proof of \eqref{dhr}$_2$ relies on the usual Taylor formulas in integral form
$$f(x\pm h)= f(x)  \pm  hf'(x) + \int_{0}^{1} (1\mp \sigma) f''(x +  \sig h)\,d\sigma.$$
Indeed, applying this identity to $\rho$,
$$
	\rho_{j\pm1} = \rho_j \pm h\partial_x\rho(x_j) +h^2 \int_{0}^{1} (1 \mp \sig)\partial_{xx}\rho(x_j+\sig h)\,d\sig,
$$
and therefore, 
$$
(\Delta_h \rho)_j = \dfrac{\rho_{j+1} -  2 \rho_j+\rho_{j-1}}{h^2}  = \int_{-1}^{1} (1-|\sig|)\partial_{xx}\rho(x_j+\sig h)\,d\sig.
$$
Since $\partial_{xx}\rho = s^2\la^2 \rho \varphi^2 (\partial_x\psi)^2 
- s\la^2 \rho \varphi (\partial_x\psi)^2 - s\la \rho \varphi \partial_{xx}\psi$, we immediately deduce \eqref{dhr}$_2$.
\end{proof}

\begin{remark}
	The coefficients of $P_h$ are intrinsically defined on the grid $\{jh\}_{j \in \llbracket 1, N\rrbracket}$ and not for $x \in [0,1]$ as formulas \eqref{A1}--\eqref{A4} may imply. But it turns out that these formula induce a natural continuous extension of these coefficients that is easier to handle. We shall therefore identify these coefficients with their continuous extension given by \eqref{A1}--\eqref{A4} without confusion.
\end{remark}

\subsection{Preliminary estimates}
Before going into the proof of the Carleman estimate itself, done in Section \ref{CARL}, we give here several key approximations on the coefficients $A_j$ defined in \eqref{A1}--\eqref{A2}--\eqref{A3}--\eqref{A4} and their derivatives. 

To begin with, we shall introduce the Landau notation $\mathcal O_{\lambda} (\epsilon)$ to denote functions $f= f(t,x)$ that satisfy, for some constant $C$ independent of $\epsilon>0$ but that might depend on $\lambda$, $|f| \leq C \epsilon $.

We are then in position to state the following basic estimates:
\begin{lemma}
	\label{rhorho}
	For all $\la>0$, $s>0$ and $h>0$ with $sh \leq 1$, for all $\sig\in[-2,2]$ and $(t,x) \in [-T,T] \times [0,1]$,
	\begin{align}
		&\dfrac{\rho(t, x+\sig h)}{\rho(t,x)} = \dfrac{e^{-s\varphi(t,x+\sig h)}}{e^{-s\varphi(t, x)}} 
		= 1 + \mathcal O_\la(sh) ~ ;
		\label{rhorho1}
	\\
		& \partial_x\left( \dfrac{\rho(t,x+\sig h)}{\rho(t,x)} \right)
			=  \mathcal O_\la(sh) ~ ;
\qquad
		\partial_t\left( \dfrac{\rho(t,x+\sig h)}{\rho(t,x)} \right)
			=  \mathcal O_\la(sh) ~ ;
		\label{rhorho3}
	\\
		& \partial_{xx}\left( \dfrac{\rho(t,x+\sig h)}{\rho(t,x)} \right)
			=  \mathcal O_\la(sh) ~ ;
			\qquad
		 \partial_{tt}\left( \dfrac{\rho(t,x+\sig h)}{\rho(t,x)} \right)
			=  \mathcal O_\la(sh).
		\label{rhorho4}
	\end{align}
\end{lemma}

\begin{proof}
Since the function $\psi$ is smooth and bounded on $(-T,T) \times (0,1)$, we have $\psi(t,x_j+\sig h) = \psi(t,x_j) + \mathcal O(h)$ and therefore
$$
	\varphi(t,x + \sigma h) = e^{\la\psi(t,x_j) +\la\mathcal O(h)} = e^{\la\psi(t,x_j)} (1+\mathcal O_\la(h)) = \varphi(t,x) + \mathcal O_{\lambda} (h).
$$
Therefore, we easily get \eqref{rhorho1} since
$$
	\rho(t,x+\sig h) = e^{ - s\varphi(t,x) + \mathcal O_\la(sh)} = \rho(t,x)(1+\mathcal O_\la(sh)).
$$
Similarly,
\begin{align*}
	\partial_x\left( \dfrac{\rho(t,x+\sig h)}{\rho(t,x)} \right)
	& =\partial_x\left( e^{-s\varphi(t,x+\sig h)}e^{s\varphi(t,x)} \right)
		\\
	& =s \lambda \left[ -\varphi(t,x+\sig h)\partial_x \psi(t,x+\sig h) +\varphi(t,x)\partial_x \psi(t,x)\right] \dfrac{e^{-s\varphi(t,x+\sig h)}}{e^{-s\varphi(t,x)}}
\end{align*}
so that 
$$
		\partial_x\left( \dfrac{\rho(t,x+\sig h)}{\rho(t,x)} \right)= s\la \mathcal{O}_\lambda (h) (1+\mathcal O_\la(sh)) 
		 =\mathcal{O}_\la(sh),
$$
which concludes the proof of \eqref{rhorho3}, left.

Of course, other estimates in \eqref{rhorho3}--\eqref{rhorho4} can be proved following the same ideas. Details are left to the reader.
\end{proof}

We can now give good approximations of the coefficients $A_j$:
\begin{lemma} \label{AA}
Set
$$
	f_1 =  \varphi \partial_x\psi, \quad  f_{2}=\varphi^2 (\partial_x\psi)^2,\quad f_{3} = \varphi (\partial_x\psi)^2, \quad f_{4} = \varphi\partial_{xx}\psi.
$$
Using the notations $A_j$ defined in Proposition~\ref{A}, for $(t,x) \in [-T,T] \times [0,1]$
and $j \in \{1, 2, 3, 4\}$, we have:

$\bullet$ On the $0$th order derivation operators:
\begin{equation}
	A_j  =  f_j + \mathcal{O}_{\lambda}(sh) 
	= m_h^+(A_j) + \mathcal{O}_{\lambda}(sh) = m_h^- (A_j) + \mathcal{O}_{\lambda}(sh)  =  m_h (A_j) + \mathcal{O}_{\lambda}(sh)~;
	\label{A0-order}
\end{equation}

$\bullet$ On the $1$st order derivation operators:
\begin{align}
	\label{A1-order-x}
	\partial_h A_j
	& =   \partial_x f_j + \mathcal{O}_\lambda(sh)= \partial_h^+ A_j  + \mathcal{O}_{\lambda}(sh) =  \partial_h^- A_j  + \mathcal{O}_{\lambda}(sh),
	\\
	\label{A1-order-t}
	\partial_t A_j
	 & = \partial_t f_j + \mathcal{O}_{\lambda}(sh)~;
\end{align}

$\bullet$ On the $2$nd order derivation operators:
\begin{equation}
	\label{A2-order}
	\Delta_h A_j =  \partial_{xx} f_j + \mathcal{O}_{\lambda}(sh), \quad  \partial_{tt} A_j = \partial_{tt} f_j + \mathcal{O}_{\lambda}(sh).
\end{equation}
\end{lemma}

\begin{proof}
Let us first notice that that all the coefficients $A_j$ can be written as
$$
	A_j (t,x)= \int_{-1}^1 \mu_j(\sigma) f_j(t,x+\sigma h)  \dfrac{e^{-s\varphi(t,x+\sig h)}}{e^{-s\varphi(t,x)}}\,d\sig, 
		\quad \mu_j (\sigma) = \left\{ 
		\begin{array}{ll} 
			1/2 &\hbox{ if } j = 1,\\ 
			(1- |\sigma|) &\hbox{ otherwise.} 
		\end{array}\right. 
$$
Using Lemma~\ref{rhorho} and the regularity of $\psi$ and $\varphi$, one can write
\begin{eqnarray*}
		A_{j}(t,x) & = &  \int_{-1}^1\mu_j(\sigma) f_j(t,x+\sigma h) \dfrac{e^{-s\varphi(t,x+\sig h)}}{e^{-s\varphi(t,x)}}\,d\sig\\
			&=& \int_{-1}^1 \mu_j(\sigma) (f_j(t,x) + \mathcal O_\lambda (h)) ( 1 + \mathcal O_\la(sh))\,d\sig
			= f_j(t,x)  + \mathcal O_\la(sh).
\end{eqnarray*}
Let us remark that it also yields the same expansion for $A_j(t,x+h)$ up to an error term of order $\mathcal{O}_\lambda(h)$, from which one easily concludes \eqref{A0-order}. 

For the first-order derivatives \eqref{A1-order-x}, we can write
\begin{eqnarray*}
	\partial_h A_j (t,x) 
	& = & \frac{1}{2} \int_{-1}^1 \partial_x A_j(t,x+ \alpha h) \, d\alpha
	\\
	& = & \frac{1}{2} \int_{-1}^1 \int_{-1}^1 \mu_j(\sigma) \partial_x f_j(t,x +(\alpha+\sigma)h) 	 \dfrac{e^{-s\varphi(t,x+(\alpha+\sig) h)}}{e^{-s\varphi(t,x+\alpha h)}} \, d\alpha d\sigma
	\\
	& &  + \frac{1}{2}  \int_{-1}^1 \int_{-1}^1 \mu_j(\sigma) f_j(t,x +(\alpha+\sigma)h) 	\partial_x\left( \dfrac{e^{-s\varphi(t,x+(\alpha+\sig) h)}}{e^{-s\varphi(t,x+\alpha h)}}\right) \, d\alpha d\sigma.
\end{eqnarray*}
But using \eqref{rhorho3}, 
$$
	\frac{1}{2}  \int_{-1}^1 \int_{-1}^1 \mu_j(\sigma) f_j(t,x +(\alpha+\sigma)h) 	\partial_x\left( \dfrac{e^{-s\varphi(t,x+(\alpha+\sig) h)}}{e^{-s\varphi(t,x+\alpha h)}}\right) \, d\alpha d\sigma
= \mathcal{O}_\lambda(sh).
$$
Therefore, we only have to estimate
$$
 	\frac{1}{2} \int_{-1}^1 \int_{-1}^1 \mu_j(\sigma) \partial_x f_j(t,x +(\alpha+\sigma)h) 	 \dfrac{e^{-s\varphi(t,x+(\alpha+\sig) h)}}{e^{-s\varphi(t,x+\alpha h)}} \, d\alpha d\sigma, 
 $$
 which can be done by using $\partial_x f_j(t,x+( \alpha+\sigma)h) = \partial_x f_j(t,x) + \mathcal{O}_\lambda(h)$, \eqref{rhorho1} and the fact that $\int_{-1}^1 \mu_j(\sigma)\, d\sigma = 1$. This yields
 $$
 	\partial_h A_j (t,x) = \partial_x f_j(t,x) + \mathcal{O}_\lambda(sh).
 $$
Of course, similar computations can be done for $\partial_h^+ A_j, \, \partial_h^- A_j$.

For \eqref{A1-order-t}, the idea is the same: we use the integral expression of the coefficients, check that the derivatives from the ratio of exponentials are of order $\mathcal{O}_\lambda(sh)$ and can therefore be neglected due to \eqref{rhorho3}, and then proceed as above. Details are left to the reader. 

For the estimates on the second order derivatives \eqref{A2-order}, this proof applies again and is therefore omitted, using this time the second order estimates \eqref{rhorho4}.
\end{proof}

To summarize the results detailed in Lemma \ref{AA}, we have proved that
$$
	A_{1} \simeq \varphi \partial_x\psi , \quad A_{2} \simeq \varphi^2 (\partial_x\psi)^2, 
	\quad A_{3} \simeq \varphi (\partial_x\psi)^2 , \quad A_{4} \simeq\varphi\partial_{xx}\psi,
$$
up to error terms in $\mathcal{O}_\lambda(sh)$, and these expressions can be differentiate twice, still with an error term of the order of $\mathcal{O}_\lambda(sh)$.\\

A more precise expression of $P_hv$ can now be deduced from Propositions \ref{B} and \ref{A}:
\begin{eqnarray*}
P_hv  &= & \partial_{tt}v - 2s\la \varphi \partial_t \psi\partial_t v+  s^2\la^2\varphi^2 \left( \partial_t \psi \right)^2 v -s\la^2\varphi\left( \partial_t \psi \right)^2v  
-~s\la\varphi (\partial_{tt} \psi) v\\
& &  -\left(1+ \dfrac {h^2}2 (s^2\la^2  A_2 - s\la^2 A_3- s\la A_4)\right)\Delta_h v  +~2 s\la A_1\partial_h v
-   (s^2\la^2  A_2 - s\la^2 A_3 - s\la A_4) v.
\end{eqnarray*}
In order to simplify the notations, we set 
$$
	A_0 = \dfrac {h^2}2 (s^2\la^2  A_2 - s\la^2 A_3- s\la A_4), 
$$
so that $P_h$ can be rewritten as
\begin{eqnarray*}
P_h v 
&= 
	&  \partial_{tt}v - 2s\la \varphi \partial_t \psi\partial_t v+  s^2\la^2\varphi^2 \left( \partial_t \psi \right)^2 v -s\la^2\varphi\left( \partial_t \psi \right)^2v  
	-s\la\varphi( \partial_{tt} \psi) v
	\\
	& &
	 -\left(1+A_0 \right)\Delta_h v  +~2 s\la A_1\partial_h v
-   (s^2\la^2  A_2 - s\la^2 A_3 - s\la A_4) v.
\end{eqnarray*}

Note that $A_0$ is expected to be small. Indeed, using Lemma~\ref{AA}, one easily gets that $A_0$ is in $\mathcal{O}_{\lambda}(sh)$ and that the same holds true for the following expressions: 
\begin{equation}\label{A0}
	A_0,m_h A_0,\, m_h^\pm A_0, \, \partial_h A_0,\, \partial_h^\pm A_0,\ \partial_t A_0,\, \Delta_h A_0, \partial_{tt} A_0 
	\quad \tn{all are } \mathcal{O}_{\lambda}(sh).
\end{equation}
We emphasize that this term $A_0$ is a purely numerical artifact which does not have any continuous counterpart. 
\subsection{Proof of the Carleman estimate}\label{CARL}

In this section, we focus on the proof of the discrete Carleman estimate \eqref{CarlemD} given in Theorem~\ref{CarlemanD}.

We first set
\begin{eqnarray}
	P_{h,1}v
		&=&
	\partial_{tt}v - \Delta_h v ( 1+ A_0) +  s^2\la^2 \left[\varphi^2 \left( \partial_t \psi \right)^2 - A_2\right]v\,,
	\label{P1}\\
	P_{h,2}v
		&=&
	- s\la^2 \left[ \varphi |\partial_t \psi|^2 - A_3\right]v 	- 2 s\la  \left[ \varphi \partial_t \psi  \partial_t v - A_1 \partial_h v\right]\, ,
	\label{P2}\\
R_hv&=&  s \la\left[\varphi \partial_{tt} \psi -A_4\right]v\, , \label{Rh}
\end{eqnarray}
so that we have
$
	P_{h,1}v+P_{h,2}v=P_h v+R_h v,
$
and in particular,
\begin{multline}\label{double}
	\int_{-T}^{T} \int_{(0,1)} | P_hv+R_hv|^2  \, dt
	=
	\int_{-T}^{T} \int_{(0,1)}| P_{h,1}v| ^2\, dt
	+
	\int_{-T}^{T} \int_{(0,1)}| P_{h,2}v| ^2 \, dt
	\\
	+
	 2\int_{-T}^{T} \int_{(0,1)}P_{h,1}vP_{h,2}v\, dt.
\end{multline}
We will then prove the following:

\begin{lemma}\label{LemDecompo}
There exist $\lambda>0$
, $s_0>0$, $\varepsilon_0 >0$ and a constant $M_0>0$ such that for all 
$s \in (s_0, \varepsilon_0/h)$, for all $v$ satisfying $v_0 = v_{N+1} = 0$ and $v(\pm T) = \partial_t v (\pm T) = 0$, 
\begin{gather}
	 s \int_{-T}^{T} \int_{(0,1)}| \partial_t v|^2 \, dt
	+  s \int_{-T}^{T} \int_{[0,1)}\ |\partial_h^+ v|^2\, dt
	+	s^3 \int_{-T}^{T} \int_{(0,1)}|v|^2\, dt
	\nonumber\\
	+ \int_{-T}^{T} \int_{(0,1)}| P_{h,1}v| ^2\, dt
	+
	\int_{-T}^{T} \int_{(0,1)}| P_{h,2}v| ^2 \, dt
	\leq 
	M_0 \int_{-T}^{T} \int_{(0,1)} |P_h v|^2 \, dt
	\label{decompo}  \\
	+
	M_0 s \int_{-T}^{T}  \left| (\partial_h^- v)_{N+1} \right|^2 \, dt
	+ 
	M_0 s \int_{-T}^{T} \int_{[0,1)} |h \partial_h^+\partial_t v|^2\, dt.
\nonumber
\end{gather}
\end{lemma}

\begin{proof} We will begin with calculating and bounding from below the product
$$
	\int_{-T}^{T} \int_{(0,1)}P_{h,1}v\, P_{h,2} v\, dt.
$$

\noindent {\bf Step 1. Explicit computations of the cross product.}\\
The proof of estimate \eqref{decompo} relies first of all on the computation of the multiplication of each term of $P_{h,1}v$ by each term of $P_{h,2}v$. We write 
$$
	\ds\int_{-T}^{T} \int_{(0,1)} P_{h,1}v \, P_{h,2}v\, dt  = \sum_{i=1}^3\sum_{j=1}^2 I_{ij}
$$ 
where $I_{ij}$ denotes the product between  the $i$-th term of $P_{h,1}$ in \eqref{P1} and the $j$-th term of $P_{h,2}$ in \eqref{P2}.
We now perform the computation of each $I_{ij}$ term. 

Of course, we shall strongly use below the properties of $v$ on the boundary and in particular that   $v(\pm T) = \partial_t v(\pm T)= 0$, 
$v_0(t)=v_{N+1}(t)= 0$ and also the fact that $\partial_t v_0(t) = \partial_t v_{N+1}(t)=0$ for all $t \in (-T,T)$. 

We shall also use the results of Lemma~\ref{AA}, which will allow us to simplify most of the expression in which the coefficients $A_j$ appear.
We recall to the reader that we will use the notation \eqref{intf} for the discrete integrals in space. In order to simplify notations, we will also omit ``$dt$" in the integrals in time (that are continuous ones). Therefore: 
$$
	\int_{-T}^{T}\int_{(0,1)} f = h \ds\sum_{j=1}^{N} \int_{-T}^{T} f_{j}(t)dt 
	\quad \tn{and } \quad
	\int_{-T}^{T}\int_{[0,1)} f = h \ds\sum_{j=0}^{N}\int_{-T}^{T} f_{j}(t)dt .
$$
%

\noindent{\it Computation of $I_{11}$.} Integrating by parts in time, 
\begin{align*}
	I_{11}
	=~&
	-s\la^2\int_{-T}^{T} \int_{(0,1)}\partial_{tt}v (\varphi |\partial_t \psi|^2-A_3)v 
	\\
	=~&
	s \la ^2 \int_{-T}^{T} \int_{(0,1)}|\partial_t v|^2(\varphi|\partial_t \psi|^2-A_3)
	-\dfrac{s\la^2}2 \int_{-T}^{T} \int_{(0,1)}|v|^2\partial_{tt}(\varphi |\partial_t \psi|^2-A_3) 
	\\
	=~&
	s \la ^2 \int_{-T}^{T} \int_{(0,1)}|\partial_t v|^2\varphi(|\partial_t \psi|^2-|\partial_x\psi|^2)
	\\
	& 
	-\dfrac{s\la^2}2 \int_{-T}^{T} \int_{(0,1)}|v|^2\partial_{tt}(\varphi |\partial_t \psi|^2-\varphi |\partial_x \psi|^2) 
	\\
	& 
	+ s \int_{-T}^{T} \int_{(0,1)}\mathcal O_\la(sh) |\partial_t v|^2 + s \int_{-T}^{T} \int_{(0,1)}  \mathcal O_\la(sh) |v|^2,
\end{align*}
using $A_3 = \varphi |\partial_x \psi|^2+  \mathcal O_{\la}(sh)$ and 
$
	\partial_{tt} A_3 =\partial_{tt}\left(\varphi |\partial_x \psi|^2\right) +  \mathcal O_{\la}(sh).
$\\

\noindent{\it Computation of $I_{12}$.} 
\begin{align*}
	I_{12}
		=~&-2s\la \int_{-T}^{T} \int_{(0,1)}\partial_{tt}v (\varphi \partial_t \psi \partial_t v -A_1 \partial_h v) 
	\\
		=~& s\la \int_{-T}^{T} \int_{(0,1)} | \partial_t v|^2\varphi \partial_{tt} \psi 
		+ s\la^2 \int_{-T}^{T} \int_{(0,1)} | \partial_t v|^2\varphi |\partial_t \psi|^2 
	\\
		&-2 s\la \int_{-T}^{T} \int_{(0,1)}  \partial_t A_1 \partial_t v \partial_h  v  -2 s\la \int_{-T}^{T} \int_{(0,1)}  A_1 \partial_t v \partial_h \partial_t v .
\end{align*}
But, by \eqref{IPP0}, 
\begin{multline*}
	-2 s\la \int_{-T}^{T} \int_{(0,1)}  A_1 \partial_t v \partial_h \partial_t v  =  s \lambda \int_{-T}^T \int_{(0,1)} |\partial_t v|^2 \partial_h A_1
		- \frac{h^2}{2} s\lambda  \int_{-T}^T \int_{(0,1)}  |\partial_h^+ \partial_t v|^2\partial_h^+ A_1.
\end{multline*}
Therefore, using Lemma~\ref{AA} for $\partial_h A_1$ and $\partial_t A_1$, we get 
\begin{align*}
	I_{12}
	=~& s\la \int_{-T}^{T} \int_{(0,1)} | \partial_t v|^2\varphi (\partial_{tt} \psi +\partial_{xx}\psi)
	+s\la^2 \int_{-T}^{T} \int_{(0,1)} | \partial_t v|^2\varphi (|\partial_t \psi|^2 + |\partial_x\psi|^2 )
	\\
	&-2s\la^2 \int_{-T}^{T} \int_{(0,1)} \varphi (\partial_t v) (\partial_h v)~\partial_t\psi~\partial_x\psi 
	-  \dfrac{s\la}2 \int_{-T}^{T} \int_{[0,1)} | h\partial_h^+\partial_t v|^2\partial_h^+ A_1 
	\\
	&
	+s \int_{-T}^{T} \int_{(0,1)}\mathcal O_\la(sh) | \partial_t v|^2 +s\int_{-T}^{T} \int_{(0,1)} \mathcal O_\la(sh) \partial_t v\partial_h v  .
\end{align*}
%

\noindent{\it Computation of $I_{21}$.}  Using \eqref{IPP1bis} and \eqref{A0},
\begin{align*}
	I_{21} =~&
		s\la^2\int_{-T}^{T} \int_{(0,1)}\Delta_h v(1+A_0) (\varphi |\partial_t \psi|^2-A_3) v 
		\\
		=~&
		-s \la^2 \int_{-T}^{T} \int_{[0,1)} |\partial_h^+ v|^2 m_h^+((1+A_0)(\varphi |\partial_t \psi|^2-A_3))
		 \\
		&+\dfrac {s \la^2}2 \int_{-T}^{T} \int_{(0,1)}  |v|^2 \Delta_h((1+A_0) (\varphi |\partial_t \psi|^2-A_3) )
		\\
		=~&-s \la^2 \int_{-T}^{T} \int_{[0,1)} |\partial_h^+ v|^2 \varphi (|\partial_t \psi|^2 - |\partial_x \psi|^2) 
		\\
		&+\dfrac {s \la^2}2 \int_{-T}^{T} \int_{(0,1)}  |v|^2 \partial_{xx} (\varphi |\partial_t \psi|^2-\varphi |\partial_x \psi|^2) 
		\\
		&+s \int_{-T}^{T} \int_{[0,1)} \mathcal O_\la(sh) |\partial_h^+v|^2 + s  \int_{-T}^T \int_{(0,1)} \mathcal{O}_\lambda(sh) |v|^2.
\end{align*}
We do not develop the term in $\partial_{xx} (\varphi |\partial_t \psi|^2-\varphi |\partial_x \psi|^2)$ since it is uniformly bounded with respect to $s$. \\

\noindent{\it Computation of $I_{22}$.}
We can split this term in two parts as follows
$$
	I_{22} = \underbrace{2s\la\int_{-T}^{T} \int_{(0,1)}\Delta_h v (1+A_0) \varphi \partial_t \psi \partial_t v}_{I_{22a}}
		 - \underbrace{2s\la\int_{-T}^{T} \int_{(0,1)}\Delta_h v (1+A_0)  A_1\partial_h v}_{I_{22b}}.
$$
To compute $I_{22a}$, we use $\Delta_h = \partial_h^- \partial_h^+$, $\partial_t v_0= \partial_t v_{N+1}=0$, and formula \eqref{dhuv} and \eqref{IPP00}:
\begin{align*}
	I_{22a} =~& - 2 s \lambda \int_{-T}^T \int_{[0,1)} (\partial_h^+ v) \partial_h^+( (1+A_0) \varphi \partial_t \psi \partial_t v)
	\\
	= ~&- 2 s \lambda \int_{-T}^T \int_{[0,1)} (\partial_h^+ v) m_h^+(\partial_t v) \partial_h^+( (1+A_0) \varphi \partial_t \psi)
	\\
	& + s \lambda \int_{-T}^T \int_{[0,1)} |\partial_h^+ v|^2 \partial_t  m_h^+( (1+A_0) \varphi \partial_t \psi).
	\\
	=~& - 2 s \lambda \int_{-T}^T \int_{[0,1)} (\partial_h^+ v) m_h^+(\partial_t v) \partial_x(\varphi \partial_t \psi)
	 + s \lambda \int_{-T}^T \int_{[0,1)} |\partial_h^+ v|^2 \partial_t  (\varphi \partial_t \psi)
	\\
	& + s  \int_{-T}^T \int_{[0,1)} \mathcal{O}_{\lambda} (sh) |\partial_h^+ v|^2  + s \int_{-T}^T \int_{[0,1)} \mathcal{O}_{\lambda} (sh)(\partial_h^+ v) m_h^+(\partial_t v).
\end{align*}
For the computation of $I_{22b}$, we rather use \eqref{IPP2} and \eqref{A0}:
\begin{align*}
	I_{22b}  =~& -s \lambda \int_{-T}^T \int_{[0,1)} |\partial_h^+ v|^2 \partial_h^+ ((1+A_0) A_1) 
	\\
	& + s \lambda \int_{-T}^T ((1+A_0) A_1)(t,1) |(\partial_h^- v)_{N+1}|^2
	\\
	& - s \lambda \int_{-T}^T ((1+A_0) A_1)(t,0) |(\partial_h^+ v)_{0}|^2
	\\
	 = ~&-s \lambda \int_{-T}^T \int_{[0,1)} |\partial_h^+ v|^2 \partial_x (\varphi \partial_x \psi) + s  \int_{-T}^T \int_{[0,1)}\mathcal{O}_{\lambda} (sh) |\partial_h^+ v|^2
	\\
	& + s \lambda \int_{-T}^T\left( [\varphi \partial_x \psi](t,1) + \mathcal{O}_\lambda(sh) \right) |(\partial_h^- v)_{N+1}|^2 
	\\
	&- s \lambda \int_{-T}^T \left( [\varphi \partial_x \psi](t,0) + \mathcal{O}_\lambda(sh) \right) |(\partial_h^+ v)_{0}|^2.
\end{align*}
Therefore, $I_{22} = I_{22a} - I_{22b}$ gives
\begin{align*}
	I_{22}  =~& 
		 - 2 s \lambda^2 \int_{-T}^T \int_{[0,1)} (\partial_h^+ v) m_h^+(\partial_t v) \varphi \partial_x\psi \partial_t \psi
		 \\
		 &
		 + s \lambda^2 \int_{-T}^T \int_{[0,1)} |\partial_h^+ v|^2  \varphi (|\partial_t \psi|^2+   |\partial_x \psi|^2)
		 \\
		 &
		 + s \lambda \int_{-T}^T \int_{[0,1)} |\partial_h^+ v|^2  \varphi (\partial_{tt} \psi +   \partial_{xx} \psi)
		\\
		& - s \lambda \int_{-T}^T\left( [\varphi \partial_x \psi](t,1) + \mathcal{O}_\lambda(sh) \right) |(\partial_h^- v)_{N+1}|^2 
		\\
		&
		+s \lambda \int_{-T}^T \left( [\varphi \partial_x \psi](t,0) + \mathcal{O}_\lambda(sh) \right) |(\partial_h^+ v)_{0}|^2
		\\
		& + s \int_{-T}^T \int_{[0,1)}  \mathcal{O}_{\lambda} (sh) |\partial_h^+ v|^2  + s \int_{-T}^T \int_{[0,1)} \mathcal{O}_{\lambda} (sh)(\partial_h^+ v) m_h^+(\partial_t v).
\end{align*}


\noindent{\it Computation of $I_{31}$.} From Lemma~\ref{AA}, one easily obtains:
\begin{align*}
	I_{31}
		=~&
		-s^3\la^4\int_{-T}^{T} \int_{(0,1)}  |v|^2(\varphi^2|\partial_t \psi|^2-A_2)(\varphi |\partial_t \psi|^2 - A_3)
	\\
		=~&
		-s^3\la^4 \int_{-T}^{T} \int_{(0,1)} |v|^2 \varphi^3 (|\partial_t \psi|^2-|\partial_x\psi|^2)^2
		+s^3 \int_{-T}^{T} \int_{(0,1)} \mathcal{O}_{\lambda} (sh)|v|^2 .
\end{align*}

\noindent{\it Computation of $I_{32}$.}
Finally, from \eqref{IPP0} of Lemma \ref{IPP}, we get
\begin{align*}
	I_{32}
		=~&
		-2s^3\la ^3\int_{-T}^{T} \int_{(0,1)}  v(\varphi^2|\partial_t \psi|^2-A_2)( \varphi \partial_t \psi \partial_t v -A_1 \partial_h v) 
		\\
		=~&
		s^3\la ^3 \int_{-T}^{T} \int_{(0,1)}|v|^2 \partial_t((\varphi^2|\partial_t \psi|^2-A_2) \varphi \partial_t \psi) 
		\\
		&
		-s^3\la ^3 \int_{-T}^{T} \int_{(0,1)}|v|^2 \partial_h(A_1(\varphi^2|\partial_t \psi|^2-A_2)) 
		\\
		&+\dfrac{s^3\la ^3}2 \int_{-T}^{T} \int_{(0,1)} |h\partial_h^+v|^2 \partial_h^+ (A_1(\varphi^2|\partial_t \psi|^2-A_2)).
\end{align*}
But, according to Lemma~\ref{AA}, 
\begin{align*}
	\partial_t((\varphi^2|\partial_t \psi|^2-A_2) \varphi \partial_t \psi) 
	=&
	~3\la \varphi^3 |\partial_t \psi|^2 \left(  |\partial_t \psi|^2 - |\partial_x \psi|^2 \right)
	\\
	&
	+3 \varphi^3 \partial_{tt}\psi  |\partial_t \psi|^2 - \varphi^3 \partial_{tt}\psi  |\partial_x \psi|^2  + \mathcal O_\la(sh)
\end{align*}
and
\begin{align}
	\partial_h(A_1(\varphi^2|\partial_t \psi|^2-A_2)) 
		=&
	 ~3\la \varphi^3 |\partial_x \psi|^2\left(|\partial_t \psi|^2 - |\partial_x \psi|^2\right) 
	 \nonumber\\
		& + \varphi^3\partial_{xx}\psi |\partial_t \psi|^2 - 3 \varphi^3 |\partial_x \psi|^2 \partial_{xx} \psi 
	+ \mathcal O_\la(sh). \nonumber
\end{align}
Thus we obtain
\begin{align*}
	I_{32}
		=~&
	3s^3\la^4 \int_{-T}^{T} \int_{(0,1)} |v|^2 \varphi^3 (|\partial_t \psi|^2-|\partial_x\psi|^2)^2
	\\
		&
	+3s^3\la ^3 \int_{-T}^{T} \int_{(0,1)}|v|^2 \varphi^3\left(|\partial_t \psi|^2 \partial_{tt} \psi +|\partial_x \psi|^2\partial_{xx} \psi  \right) 
	\\
		&-s^3\la ^3 \int_{-T}^{T} \int_{(0,1)}|v|^2 \varphi^3\left( |\partial_x \psi|^2 \partial_{tt} \psi + |\partial_t \psi|^2\partial_{xx} \psi\right) 
	\\
		&
	+s\int_{-T}^{T} \int_{(0,1)} \mathcal{O}_{\lambda}(sh) |\partial_h^+v|^2 +s^3 \int_{-T}^{T} \int_{(0,1)}  \mathcal{O}_{\lambda}(sh)|v|^2.
\end{align*}

\noindent{\it Final computation.}
Gathering all the terms, one can write
\begin{equation}
	\int_{-T}^{T} \int_{(0,1)} P_{h,1}v\,  P_{h,2}v  = I_v + I_{\partial_t, \nabla v}  + I_{\{0,1\}}+I_\tn{Tych}, 
\end{equation}
where $I_v$ contains all the terms in $|v|^2$:
\begin{equation}
	I_v
	 = 
	\int_{-T}^T \int_{(0,1)} |v|^2 F ,
\end{equation}
with $F$ given by
\begin{align*}
	F =~& 
	-\dfrac{s\la^2}2 \partial_{tt}(\varphi |\partial_t \psi|^2-\varphi |\partial_x \psi|^2) 
	+\dfrac {s \la^2}2 \partial_{xx} (\varphi |\partial_t \psi|^2-\varphi |\partial_x \psi|^2)
	\\
	& 
	+ s^3 \lambda^3  \varphi^3 (|\partial_t \psi|^2 - |\partial_x \psi|^2)(\partial_{tt} \psi - \partial_{xx} \psi)
	 +2 s^3 \lambda^3 \varphi^3 (|\partial_t \psi|^2 \partial_{tt} \psi + |\partial_x \psi|^2 \partial_{xx} \psi) 
	\\
	&
	+ 2 s^3 \lambda^4 \varphi^3 (|\partial_t \psi^2 - |\partial_x \psi|^2)^2 + s^3 \mathcal{O}_\lambda(sh)
	\\
	 = ~&
	 s^3 \lambda^3  \varphi^3 (|\partial_t \psi|^2 - |\partial_x \psi|^2)(\partial_{tt} \psi - \partial_{xx} \psi)
	 +2  s^3 \lambda^3 \varphi^3 (|\partial_t \psi|^2 \partial_{tt} \psi + |\partial_x \psi|^2 \partial_{xx} \psi) 
	\\
	&
	+ 2 s^3 \lambda^4 \varphi^3 (|\partial_t \psi|^2 - |\partial_x \psi|^2)^2 + s^3 \mathcal{O}_\lambda(sh) + s \mathcal{O}_\lambda(1)~;
\end{align*}
$I_{\partial_t, \nabla v}$ contains all the terms involving first order derivatives of $v$:
\begin{align*}
	\lefteqn{I_{\partial_t, \nabla v} 
		=  2 s \la ^2 \int_{-T}^{T} \int_{(0,1)}|\partial_t v|^2\varphi|\partial_t \psi|^2
		+ 2 s \la ^2 \int_{-T}^{T} \int_{[0,1)}|\partial_h^+ v|^2\varphi |\partial_x \psi|^2
	}\\
		&- 2s\la^2 \int_{-T}^{T} \int_{(0,1)} \varphi (\partial_t v) (\partial_h v)~\partial_t\psi~\partial_x\psi 
	- 2 s \lambda^2 \int_{-T}^T \int_{[0,1)}\varphi (\partial_h^+ v) m_h^+(\partial_t v) \partial_x \psi \partial_t \psi
	\\
		&+s\la \int_{-T}^{T} \int_{(0,1)}| \partial_t v|^2 \varphi ( \partial_{tt} \psi + \partial_{xx}\psi)
	+s\la \int_{-T}^{T} \int_{[0,1)}| \partial_h^+ v|^2 \varphi ( \partial_{tt} \psi + \partial_{xx}\psi)
	\\
		&+ s\int_{-T}^T \int_{(0,1)} \Big( \mathcal{O}_{\lambda}(sh) |\partial_t v|^2 +  \mathcal{O}_{\lambda}(sh)\partial_t v\partial_h v \Big)
	\\
	&
	+s\int_{-T}^T \int_{[0,1)} \Big( \mathcal{O}_{\lambda}(sh)|\partial_h^+ v|^2 +   \mathcal{O}_{\lambda}(sh)m_h^+(\partial_t v) \partial_h^+ v \Big)~;		
\end{align*}
$I_{\{0,1\}}$ contains all the boundary terms:
\begin{align*}
	I_{\{0,1\}}	=
		& - s \lambda \int_{-T}^T\left( [\varphi \partial_x \psi](t,1) + \mathcal{O}_\lambda(sh) \right) |(\partial_h^- v)_{N+1}|^2 
		\\
		&
		+s \lambda \int_{-T}^T \left( [\varphi \partial_x \psi](t,0) + \mathcal{O}_\lambda(sh) \right) |(\partial_h^+ v)_{0}|^2;
\end{align*}
$I_\tn{Tych}$ contains the term corresponding to the Tychonoff regularization:
\begin{equation}
	I_\tn{Tych} = -  \dfrac{s\la}2 \int_{-T}^{T} \int_{[0,1)} | h\partial_h^+\partial_t v|^2\partial_h^+ A_1.
\end{equation}

\noindent {\bf Step 2. Bounding each term from below.}\\
In the sequel, $M>0$ and $C>0$ will denote generic constants that all are independent of $h$ and~$s$ but may depend on $\lambda$.\\
\noindent {\it Step 2.1. Dealing with the $0$ order terms in $v$.} \\
We have:
$$
	F=  s^3 \lambda^3 \varphi^3 G +  s^3 \mathcal{O}_\lambda(sh) + s \mathcal{O}_\lambda(1),
$$
with 
\begin{align*}
	G
	 = ~&2 \lambda  (|\partial_t \psi|^2 - |\partial_x \psi|^2)^2+ (|\partial_t \psi|^2 - |\partial_x \psi|^2) (\partial_{tt} \psi - \partial_{xx} \psi)
	 +2  (|\partial_t \psi|^2 \partial_{tt} \psi + |\partial_x \psi|^2 \partial_{xx} \psi) 
	\\
	 = ~&2 \lambda  (|\partial_t \psi|^2 - |\partial_x \psi|^2)^2+ (|\partial_t \psi|^2 - |\partial_x \psi|^2) (\partial_{tt} \psi - \partial_{xx} \psi)
	 +2  (|\partial_t \psi|^2 - |\partial_x \psi|^2) \partial_{tt} \psi 
	 \\
	& 
+2  |\partial_x \psi|^2 ( \partial_{tt} \psi+ \partial_{xx} \psi ) .
\end{align*}
But  $\partial_{tt} \psi+ \partial_{xx} \psi = 2(1- \beta)>0$ and $\inf_{(0,1)}  |\partial_x \psi|^2 = 4 \inf_{(0,1)}  |x-x^0|^2$ is strictly positive since $x^0 \notin [0,1]$. Therefore, setting $X= |\partial_t \psi|^2 - |\partial_x \psi|^2$, we have
$$
	G \geq 2 \lambda X^2 - 2X(3 \beta +1)  + c, \quad \hbox{ with } c =  16(1- \beta) \inf_{(0,1)}  |x-x^0|^2>0.
$$
Thus, there exists $\lambda_0>0$ large enough such that for all $\lambda \geq \lambda_0$, this expression can be made strictly positive.
Therefore, for $\lambda \geq \lambda_0$, we get a positive constant $c_*>0$ independent of $\lambda$ such that 
$$
	I_v \geq 2 c_* s^3 \lambda^3 \int_{-T}^T \int_{(0,1)} \varphi^3 |v|^2 - (s^3 \mathcal{O}_\lambda(sh) + s \mathcal{O}_\lambda(1) ) \int_{-T}^T \int_{(0,1)}  |v|^2.
$$
Thereby, bounding $\varphi$ from below by $1$, we can choose $s_0(\lambda)$ such that for all $s \geq s_0(\lambda)$,
\begin{equation}\label{order0}
	I_v \geq  c_* s^3 \lambda^3 \int_{-T}^T \int_{(0,1)} |v|^2 - s^3 \mathcal{O}_\lambda(sh)  \int_{-T}^T \int_{(0,1)}  |v|^2.
\end{equation}
From then on, we fix $\lambda \geq \lambda_0$.

\noindent {\it Step 2.2. Dealing with the first-order derivatives.}\\
The idea is to show that the terms in which $s\lambda^2$ appears are positive up to a small error term, and then check that the terms in $s\lambda$ are strictly positive.

On the one hand, 
\begin{multline*}
		\left| \int_{-T}^{T} \int_{(0,1)} \varphi (\partial_t v) (\partial_h v)~\partial_t\psi~\partial_x\psi \right|
	\leq
	 \dfrac 12 \int_{-T}^{T} \int_{(0,1)} \varphi |\partial_t v|^2|\partial_t \psi|^2  
	+ 
	\dfrac 12 \int_{-T}^{T} \int_{(0,1)} \varphi  |\partial_h v|^2  |\partial_x \psi|^2 
	\\
	\leq
	\dfrac 12 \int_{-T}^{T} \int_{(0,1)} \varphi |\partial_t v|^2|\partial_t \psi|^2  
	+
	\dfrac 12 \int_{-T}^{T} \int_{[0,1)} \varphi  |\partial_h^+ v|^2  |\partial_x \psi|^2 
	+ 
	\mathcal O_{\lambda}(sh) \int_{-T}^{T} \int_{[0,1)}  |\partial_h^+ v|^2
\end{multline*}
and similarly, 
\begin{eqnarray*}
	\lefteqn{
		\left| \int_{-T}^{T} \int_{[0,1)} \varphi  (\partial_h^+ v) ~m_h^+(\partial_t v) ~\partial_t\psi~\partial_x\psi \right|
		}
		\\
	&\leq& 
	\dfrac 12 \int_{-T}^{T} \int_{[0,1)} \varphi  |\partial_h^+ v|^2  |\partial_x \psi|^2 
	+
	\dfrac 12 \int_{-T}^{T} \int_{[0,1)} \varphi  |m_h^+(\partial_t v)|^2  |\partial_t \psi|^2
		\\
	&\leq& 
	\dfrac 12 \int_{-T}^{T} \int_{[0,1)} \varphi  |\partial_h^+ v|^2  |\partial_x \psi|^2 
	+
	\dfrac {1}2 \int_{-T}^{T} \int_{(0,1)} \varphi  |\partial_t v|^2  |\partial_t \psi|^2   
	+ \mathcal O_{\lambda}(sh) \int_{-T}^{T} \int_{(0,1)}  |\partial_t v|^2.
\end{eqnarray*}
Therefore, 
\begin{multline*}
		 2 s \la ^2 \int_{-T}^{T} \int_{(0,1)}|\partial_t v|^2\varphi|\partial_t \psi|^2
		+ 2 s \la ^2 \int_{-T}^{T} \int_{[0,1)}|\partial_h^+ v|^2\varphi |\partial_x \psi|^2
	\\
		-2s\la^2 \int_{-T}^{T} \int_{(0,1)} \varphi (\partial_t v) (\partial_h v)~\partial_t\psi~\partial_x\psi 
		-2 s \lambda^2 \int_{-T}^T \int_{[0,1)}\varphi (\partial_h^+ v) m_h^+(\partial_t v) \partial_x \psi \partial_t \psi
	\\
	 	  \geq  - ~s \mathcal O_{\lambda}(sh) \int_{-T}^{T} \int_{[0,1)}  |\partial_h^+ v|^2
		 -  s \mathcal{O}_\lambda(sh) \int_{-T}^{T} \int_{(0,1)}  |\partial_t v|^2. 
\end{multline*}

On the other hand, focusing on the terms in $s \lambda$, we have $ \partial_{tt} \psi + \partial_{xx}\psi  = 2(1- \beta) >0$, and then, bounding $\varphi = e^{\lambda \psi}$ from below by $1$, we obtain:
\begin{align}
	I_{\partial_t, \nabla v} 
		\geq ~&
	2 s \lambda (1- \beta) \int_{-T}^T \int_{[0,1)} |\partial_h^+ v|^2 + 2 s \lambda (1- \beta) \int_{-T}^T \int_{(0,1)} |\partial_t v|^2 
		\nonumber\\
		& -
	s \mathcal O_{\lambda}(sh) \int_{-T}^{T} \int_{[0,1)}  |\partial_h^+ v|^2 -  s \mathcal{O}_\lambda(sh) \int_{-T}^{T} \int_{(0,1)}  |\partial_t v|^2,\label{order1}
\end{align}
where we used that, by Cauchy Schwartz, 
$$
	\int_{-T}^T \int_{(0,1)} |\partial_h v |^2 \leq  \int_{-T}^T \int_{[0,1)} |\partial_h^+ v|^2,
	\quad
	\int_{-T}^T \int_{[0,1)}  |m_h^+(\partial_t v)|^2 \leq  \int_{-T}^T \int_{(0,1)} |\partial_t v|^2.
$$

\noindent {\it Step 2.3. The boundary terms.} \\
Since $\min_{(-T, T) \times (0,1)} \varphi \partial_x \psi  >0$ (recall $x^0 <0$), then there exists $\varepsilon_1 (\lambda)>0$ such that taking $sh \leq \varepsilon_1(\lambda)$, 
$$
	\mathcal{O}_\lambda(sh) \leq \min_{(t,x) \in (-T, T) \times (0,1)}\left\{ \varphi(t,x) \partial_x \psi(t,x)\right\}.
$$

Therefore,
\begin{align}
	I_{\{0,1\}}	=
		& - s \lambda \int_{-T}^T\left( [\varphi \partial_x \psi](t,1) + \mathcal{O}_\lambda(sh) \right) |(\partial_h^- v)_{N+1}|^2 
		+s \lambda \int_{-T}^T \left( [\varphi \partial_x \psi](t,0) + \mathcal{O}_\lambda(sh) \right) |(\partial_h^+ v)_{0}|^2
		\nonumber\\
		\geq 
		& - 2 s \lambda \int_{-T}^T [\varphi \partial_x \psi](t,1) |(\partial_h^- v)_{N+1}|^2
		\geq 
		 -  s C_\lambda \int_{-T}^T |(\partial_h^- v)_{N+1}|^2.
		 \label{boundary}
\end{align}

\noindent {\it Step 2.4. The Tychonoff regularization.}\\
Let us recall that 
$\partial_h^+ A_1 = \lambda \varphi |\partial_x \psi|^2 + \varphi \partial_{xx} \psi + \mathcal{O}_\lambda(sh) 
= \mathcal{O}_\lambda(1) $ since $sh \leq \varepsilon(\lambda)$. Thus
\begin{equation}\label{tych}
	I_\tn{Tych}  = -  \dfrac{s\la}2 \int_{-T}^{T} \int_{[0,1)} | h\partial_h^+\partial_t v|^2\partial_h^+ A_1
	 \geq -  s C_{\lambda}  \int_{-T}^{T} \int_{[0,1)} | h\partial_h^+\partial_t v|^2.
\end{equation}
Noticing that $ \lambda \varphi |\partial_x \psi|^2 + \varphi \partial_{xx} \psi>0$,  $I_\tn{Tych}\leq 0$ and cannot be made positive. This is not only a technical matter, since otherwise we would get uniform observability results for the semidiscrete wave equations, which is proved not to hold in \cite{InfZua}.\\

\noindent {\bf Step 3. Proof of Lemma \ref{LemDecompo}.}\\
Recall that $\lambda$ is fixed from Step 2.1.\\
Collecting the results \eqref{order1}-\eqref{tych},  of Step 2, we have proved that for $s\geq s_0(\lambda)$ and $sh \leq \varepsilon_1(\lambda)$,
\begin{eqnarray*}
	\int_{-T}^{T} \int_{(0,1)} P_{h,1}v \,P_{h,2}v
	 &\geq&~ 
	2 s \lambda (1- \beta) \int_{-T}^T \int_{[0,1)} |\partial_h^+ v|^2 
	+ ~2 s \lambda (1- \beta) \int_{-T}^T \int_{(0,1)} |\partial_t v|^2 
	\\
	&& 
	 - 	s \mathcal O_{\lambda}(sh) \int_{-T}^{T} \int_{[0,1)}  |\partial_h^+ v|^2 
	- ~ s \mathcal{O}_\lambda(sh) \int_{-T}^{T} \int_{(0,1)}  |\partial_t v|^2
		\\
	&&	
	+ c_* s^3 \lambda^3 \int_{-T}^T \int_{(0,1)}  |v|^2
	 - ~s^3 \mathcal{O}_\lambda(sh)  \int_{-T}^T \int_{(0,1)}  |v|^2
	\\
		&&
	 - s C_\lambda \int_{-T}^T  |(\partial_h^- v)_{N+1}|^2
			- ~ sC_\lambda  \int_{-T}^{T} \int_{[0,1)} | h\partial_h^+\partial_t v|^2. 
\end{eqnarray*}
Therefore, taking $sh$ small enough such that 
$$
	\mathcal{O}_\lambda(sh) \leq\min\left\{ \lambda (1- \beta)  , \frac{c_* \lambda^3}{2}, \varepsilon_1(\lambda)  \right\}, 
$$
which defines $\varepsilon(\lambda)>0$, we obtain, for some constant $M_1= M_1(\lambda)>0$,
\begin{multline}
	 s \int_{-T}^{T} \int_{(0,1)}| \partial_t v|^2 
	+  s \int_{-T}^{T} \int_{[0,1)}\ |\partial_h^+ v|^2 
	+	s^3 \int_{-T}^{T} \int_{(0,1)}|v|^2
	\\
	\leq 
	M_1 \int_{-T}^{T} \int_{(0,1)} P_{1,h} v \, P_{2,h} v
	+
	M_1 s \int_{-T}^{T}  \left| (\partial_h^- v)_{N+1} \right|^2  
	\label{decompo-0}  
	+ 
	M_1 s \int_{-T}^{T} \int_{[0,1)} |h \partial_h^+\partial_t v|^2 .
\end{multline}
Now, we use that from \eqref{double}, 
\begin{multline*}
	2 \int_{-T}^{T} \int_{(0,1)} P_{1,h} v \, P_{2,h} v + \int_{-T}^{T} \int_{(0,1)} |P_{1,h} v|^2 + \int_{-T}^{T} \int_{(0,1)} |P_{1,h} v|^2 
		\\
	\leq
	2 \int_{-T}^{T} \int_{(0,1)} |P_{h} v|^2 +  2 \int_{-T}^{T} \int_{(0,1)} |R_h v|^2,
\end{multline*}
where $R_h v$ is given by \eqref{Rh}, which yields, for some $M_2(\lambda)>0$,
\begin{multline}
	 s \int_{-T}^{T} \int_{(0,1)}| \partial_t v|^2 
	+  s \int_{-T}^{T} \int_{[0,1)}\ |\partial_h^+ v|^2
	+	s^3 \int_{-T}^{T} \int_{(0,1)}|v|^2
	\\
	+ \int_{-T}^{T} \int_{(0,1)}| P_{h,1}v| ^2
	+
	\int_{-T}^{T} \int_{(0,1)}| P_{h,2}v| ^2 
	\leq 
	M_2 \int_{-T}^{T} \int_{(0,1)} |P_h v|^2 
	\\
	+
	M_2 \int_{-T}^{T} \int_{(0,1)} |R_h v|^2 
	\label{decompo-00}  
	+
	M_2 s \int_{-T}^{T}  \left| (\partial_h^- v)_{N+1} \right|^2  
	+ 
	M_2 s \int_{-T}^{T} \int_{[0,1)} |h \partial_h^+\partial_t v|^2.
\end{multline}
Therefore, since
$$
	 \int_{-T}^{T} \int_{(0,1)} |R_h v|^2 \leq  s^2 \lambda^2  \int_{-T}^{T} \int_{(0,1)} |v|^2 \left(\varphi \partial_{tt} \psi - A_4 \right)^2
	 \leq s^2 \mathcal{O}_\lambda(1) \int_{(0,1)} |v|^2,
$$
this term can be absorbed by the left hand side of \eqref{decompo-00} by taking $s$ large enough, thus yielding to \eqref{decompo}.
\end{proof}

\begin{proof}[Proof of Theorem \ref{CarlemanD}]
The Carleman estimate \eqref{CarlemD} of Theorem \ref{CarlemanD} will now be deduced from Lemma \ref{LemDecompo}. Actually, it simply consists in writing \eqref{decompo} in terms of $w$ instead of $v$, using that $w = v \exp(- s \varphi)$ and, by construction, see \eqref{PhDef}, $\exp(s \varphi) L_h w = P_h v$.

In particular, we have
\begin{align*}
	e^{s \varphi} \partial_t w 
		& = e^{s \varphi} \partial_t (v e^{-s \varphi} )= \partial_t v - s \lambda v \varphi \partial_t \psi ,
	\\
	e^{s \varphi} \partial_h^+ w
		& = e^{s \varphi} \partial_h^+ (v e^{-s \varphi}) = \partial_h^+ v (e^{s\varphi} m_h^+(e^{-s\varphi})) - e^{ s\varphi} \partial_h^+ (e^{-s\varphi} ) m_h^+ v, 
\end{align*}
and, since, similarly as in Lemma \ref{rhorho},
$$
	e^{s\varphi} m_h^+(e^{-s\varphi}) = 1 + \mathcal{O}_\lambda(sh) \hbox{ and }  e^{ s\varphi} \partial_h^+ (e^{-s\varphi} ) = - s \lambda \varphi  \partial_x \psi + \mathcal{O}_\lambda(sh), 
$$
we get, for $sh$ small enough,
\begin{align*}
	e^{2 s \varphi} |\partial_t w |^2
		& \leq 2 |\partial_t v|^2 + 2s^2 \lambda^2 \varphi^2 |\partial_t \psi|^2 |v|^2 ,
	\\
	e^{2s \varphi} |\partial_h^+ w|^2
		& \leq 3 | \partial_h^+ v|^2+ 3 s^2 \lambda^2 \varphi^2 |\partial_x \psi|^2  |m_h^+ v|^2.
\end{align*}
Therefore, there exists $M_3(\lambda)$ such that
\begin{multline}
	 s \int_{-T}^{T} \int_{(0,1)}e^{2 s \varphi} | \partial_t w|^2  
	+  
	s \int_{-T}^{T} \int_{[0,1)}e^{2 s\varphi} |\partial_h^+ w|^2
	+
	s^3 \int_{-T}^{T} \int_{(0,1)} e^{2 s \varphi}  |w|^2
	\\
	\leq 
	M_3 s \int_{-T}^{T} \int_{(0,1)}| \partial_t v|^2  
	+  
	M_3 s \int_{-T}^{T} \int_{[0,1)}\ |\partial_h^+ v|^2 
	+
	M_3 s^3 \int_{-T}^{T} \int_{(0,1)}|v|^2.
	\label{DistributedTermsV-W}
\end{multline}

Now, it remains to estimate the right hand side of \eqref{decompo} in terms of $w$. For the boundary term, we use the fact that $\varphi(t,\cdot)$ is increasing and therefore
\begin{equation}
	\label{BoundaryTerms}
	|(\partial_h^- v)_{N+1}| = \left| \frac{v_N}h\right|  =  \left| \frac{w_N e^{s \varphi_N}}h\right| \\
	\leq  \left| \frac{w_N}h\right| e^{s \varphi(t,1)} = |(\partial_h^- w)_{N+1}| e^{s \varphi(t,1)}.
\end{equation}
Finally, we bound the term corresponding to the Tychonoff regularization, using \eqref{dhuv} of Lemma~\ref{IPP} and  
$\partial_t v = e^{s\varphi}\partial_t w + w \partial_t(e^{s\varphi})$: 
\begin{align*}
	h \partial_h^+ \partial_t v 
		&= 
	h \partial_h^+ (\partial_t w) m_h^+(e^{s \varphi}) + h (\partial_h^+ w) m_h^+ \partial_t (e^{s \varphi}) 
	\\
	&\quad + m_h^+ (\partial_t w) h \partial_h^+ (e^{s \varphi}) + (m_h^+ w) h \partial_h^+ (\partial_t e^{s \varphi}).
\end{align*}
Again, similarly to Lemma \ref{rhorho}, one can prove
$$
\frac{h m_h^+(\partial_t(e^{s\varphi}))}{e^{s\varphi}} = \mathcal{O}_\lambda(sh), 
\quad
\frac{h \partial_h^+(e^{s\varphi})}{e^{s\varphi}} = \mathcal{O}_\lambda(sh),
\quad \tn{and } \quad 
\frac{h \partial_h^+(\partial_t e^{s\varphi})}{e^{s\varphi}} = s \mathcal{O}_\lambda(sh),
$$
and deduce a bound of the form:
\begin{align*}
	|h \partial_h^+ \partial_t v |
		&\leq 
	 |h \partial_h^+ (\partial_t w)| e^{s \varphi} (1 +  \mathcal{O}_\lambda(sh)) + \mathcal{O}_\lambda(sh)  |\partial_h^+ w| e^{s \varphi}
	\\
	& 
	\quad + \mathcal{O}_\lambda(sh) |m_h^+ (\partial_t w)| e^{s \varphi} + s \mathcal{O}_\lambda(sh) |m_h^+ w| e^{s \varphi}.
\end{align*}
Therefore, one gets:
\begin{align}
	 \lefteqn{
	 s \int_{-T}^{T} \int_{[0,1)} |h \partial_h^+\partial_t v|^2 
	}
	\nonumber\\
	 \leq &  ~M_4 s  \int_{-T}^{T} \int_{[0,1)} e^{2 s\varphi} |h \partial_h^+\partial_t w|^2
	+ M_4 s  \mathcal{O}_\lambda(sh) \int_{-T}^{T} \int_{[0,1)}e^{2 s \varphi}\left(  |\partial_h^+ w|^2 + |m_h^+ (\partial_t w)|^2\right)
	\nonumber	\\
	& 
	+M_4 s^3  \mathcal{O}_\lambda(sh) \int_{-T}^{T} \int_{[0,1)}e^{2 s \varphi}  |m_h^+ w|^2 
	\nonumber	\\
	\leq &
	~M_5 s \int_{-T}^{T} \int_{[0,1)} e^{2 s\varphi} |h \partial_h^+\partial_t w|^2 
	\label{TycV-W}
	 + M_5 s  \mathcal{O}_\lambda(sh) \int_{-T}^{T} \int_{[0,1)}e^{2 s \varphi}  |\partial_h^+ w|^2 
	\\
	&
	+ M_5 s  \mathcal{O}_\lambda(sh) \int_{-T}^{T} \int_{(0,1)}e^{2 s \varphi}  |\partial_t w|^2 
	+M_5 s^3  \mathcal{O}_\lambda(sh) \int_{-T}^{T} \int_{(0,1)}e^{2 s \varphi}  |w|^2 .
	\nonumber
\end{align}

Hence, combining \eqref{DistributedTermsV-W}--\eqref{BoundaryTerms}--\eqref{TycV-W}, plugging them in \eqref{decompo} and choosing $sh$ small enough so that $M_0 M_3 M_5 \mathcal{O}_\lambda(sh) \leq 1/2$ and $s h \leq \varepsilon_0$ (given by Lemma \ref{LemDecompo}), we obtain the desired Carleman estimate~\eqref{CarlemD}.
\end{proof}

\section{Uniform stability estimates}\label{InversePb}

In this section, we state and prove uniform stability results for the semi-discrete framework, announced by \eqref{UniformStability-2} in the introduction. 

\subsection{Statements of the results}

Similarly to Theorem~\ref{TCWE1}, we will prove the following local stability result:
\begin{theorem}\label{TCWE2}
Let $m>0$, $K>0$, $r>0$, $T>1$ and $p_h \in L^\infty_{h,\leq m}(0,1)$.

Consider the equation
$$
	\left\{\begin{array}{ll}
		\partial_{tt}y_{j,h} - \left(\Delta_h y_{h}\right)_j+p_{j,h} y_{j,h}=g_{j,h}, \quad & t\in (0,T), \ j\in \llbracket 1,N \rrbracket,\\
		y_{0,h}(t)=g_h^0(t), \quad y_{N+1,h}(t)=g_h^1(t),& t\in (0,T),\\
		y_{j,h}(0)= y_{j,h}^0, \quad \partial_t y_{j,h}(0)= y_{j,h}^1, &  j\in \llbracket 1,N \rrbracket,
	\end{array}\right.
$$
and assume that
\begin{equation}\label{RegDiscrete}
	\norm{ y[p_h]}_{H^1 (0,T;L_h^{\infty}(0,1)) } \leq K
\end{equation}
and
\begin{equation}
	\inf_{ j \in \llbracket 1, N\rrbracket} |y^{0}_{j,h}|\geq r.
\end{equation}
Then there exists a constant $C=C(T,m, K, r)>0$ independent of $h$ such that for all $ q_h\in L^\infty_{h,\leq m}(0,1)$, the uniform stability estimate \eqref{UniformStability-2} holds:
\begin{eqnarray}
	\norm{q_h - p_h}_{L^2_h(0,1)} &\leq& C \norm{\partial_t(\partial_h^-y_h)_{N+1}[q_h] - \partial_t(\partial_h^-y_h)_{N+1}[p_h]}_{L^2(0,T)} 
	\label{UniformStability-3}\\
	&&\hfill{}+ C \norm{h \partial_h^+\partial_{tt} y_h[q_h] - h \partial_h^+\partial_{tt} y_h[p_h]}_{L^2(0,T;L^2_h[0,1))}.\nonumber
\end{eqnarray}
\end{theorem}

Before giving the proof of Theorem \ref{TCWE2} at the end of this section, we will begin by a stability theorem for the following inverse source problem, using the discrete Carleman estimate obtained in the previous section:

\begin{theorem}\label{TCWE3}
Let $m>0$, $K>0$, $r>0$, $T>1$.

Let $f_h\in L_h^2(0,1)$, $R_h\in H^1(0,T;L_h^{\infty}(0,1))$ such that
\begin{equation}
	\label{DiscreteSource}
	\norm{R_h}_{H^1(0,T;L_h^{\infty}(0,1)) } \leq K\quad \hbox{ and } \inf_{ j \in \llbracket 1, N\rrbracket} |R_{j,h}(0)|\geq r.
\end{equation}

Let $q_h \in L^\infty_{h,\leq m}(0,1)$ and consider the semi-discrete wave equation
\begin{equation}
	\left\{\begin{array}{ll}
		\label{SDWE2}
			\partial_{tt}u_{j,h} - \left(\Delta_h u_h\right)_j+q_{j,h} u_{j,h}=f_{j,h}R_{j,h}(t), & t\in (0,T), j\in \llbracket 1,N \rrbracket,\\
			u_{0,h}(t)=0, \quad u_{{N+1},h}(t)=0,& t\in (0,T),\\
			u_{j,h}(0)= 0, \quad\partial_t u_{j,h}(0)= 0, & j\in \llbracket 1,N \rrbracket.
	\end{array}\right.
\end{equation}
Then there exists a constant $C=C(T,m,K,r)>0$ independent of $h$ and such that 
\begin{align}
	 & \norm{ \partial_t (\partial_h^-u_h)_{N+1}}_{L^2(0,T)}
	 +  \norm{h\partial_h^+ \partial_{tt} u_h}_{L^2(0,T;L_h^2([0,1)))} 
	  \leq  C \norm{ f_h} _{L_h^2(0,1)},
	 \label{DirectEstimateUnif}
	\\
	&
	\norm{ f_h} _{L_h^2(0,1)}
	\leq
	C \norm{ \partial_t (\partial_h^-u_h)_{N+1}}_{L^2(0,T)}
	 + C \norm{h\partial_h^+ \partial_{tt} u_h}_{L^2(0,T;L_h^2([0,1)))}.
	 \label{ReverseEstimateUniform}
\end{align}
\end{theorem}
Theorem \ref{TCWE2} will then be a simple consequence of Theorem \ref{TCWE3}, see its proof in Section \ref{ProofThmTCWE2}.

\subsection{Stability for the inverse source problem}

Before going into the proof of Theorem \ref{TCWE3}, we first recall some counterparts of the classical energy estimates for the solutions of the continuous wave equation in the context of the semi-discrete wave equation:

\begin{lemma}\label{LemEnergy}
Let $ m >0$ and $q_h \in L^\infty_{h, \leq m} (0,1)$.

Let $g_h \in L^1(0,T;L_h^2(0,1))$ and $(z_h^0, z_h^1)$ be discrete functions and let  $z_h$ be the solution of 
$$
	\left\{\begin{array}{ll}
			\partial_{tt}z_{j,h} - \left(\Delta_h z_h\right)_j+q_{j,h} z_{j,h}=g_{j,h}(t), & t\in (0,T), j\in \llbracket 1,N \rrbracket,\\
			z_{0,h}(t)=0, \quad z_{{N+1},h}(t)=0,& t\in (0,T),\\
			z_{j,h}(0)= z_{j,h}^0, \quad\partial_t z_{j,h}(0)= z_{j,h}^1, & j\in \llbracket 1,N \rrbracket.
	\end{array}\right.
$$
Then, setting
\begin{equation}
	\label{DefEnergy-H}
	E^z_h(t) = \norm{\partial_h^+ z_h(t)}^2_{L_h^2([0,1))} +  \norm{\partial_t z_h(t)}^2_{L_h^2(0,1)} +  \norm{z_h(t)}^2_{L^2_h(0,1)},
\end{equation}
there exists a constant $C = C(T,m)>0$ independent of $h$ and such  for all $t\in(0,T)$,
\begin{equation}
	\label{Energy-Uniform}
 	E^z_h(t) \leq C\left(\norm{\partial_h^+ z^0_h}^2_{L_h^2([0,1))} +  \norm{z^1_h}^2_{L_h^2(0,1)} + \norm{g_h}^2_{L^1(0,T;L^2_h(0,1))}\right).
\end{equation}
We also have the following ``hidden regularity'' property:
\begin{equation}
	\label{HiddenRegularityUniform}
	\norm{ (\partial_h^- z_h)_{N+1} }_{L^2(0,T)}^2 
	\leq C \left(\norm{\partial_h^+ z^0_h}^2_{L_h^2([0,1))} +  \norm{z^1_h}^2_{L_h^2(0,1)} + \norm{g_h}^2_{L^1(0,T;L^2_h(0,1))}\right).
\end{equation}
\end{lemma}

\begin{proof}
	The proof is somewhat classical, except perhaps for \eqref{HiddenRegularityUniform}. The quantity $E^z_h$ is usually called the discrete energy of the solution $z_h$. We sketch it for the convenience of the reader since it will be useful in the sequel.
	
	Differentiating $E_h^z$ with respect to the time $t$, we obtain
	\begin{align*}
		\frac{dE_h^z}{dt} (t)
		&= 2 \int_{(0,1)} \left( 	\partial_{tt}z_h - \Delta_h z_h+ z_h \right) \partial_t z_h
		\\
		& \leq 2 \int_{(0,1)} |g_h(t) \partial_t z_h| + 2 (m+1) \int_{(0,1)} |z_h \partial_t z_h|
		\\
		& \leq 2 \left(\int_{(0,1)} |g_h(t)|^2\right)^{1/2} \sqrt{E_h^z(t)} + (m+1) E_h^z(t).
	\end{align*}
	Therefore, 
	$$
		\frac{d \sqrt{E_h^z}}{dt} \leq \left(\int_{(0,1)} |g_h(t)|^2\right)^{1/2}+ \frac{(m+1)}2 \sqrt{E_h^z}
	$$
	and Gronwall's estimate then yields a constant $C(T, m)$ such that for all $t \in(0,T)$, 
	\begin{equation}
	\label{Energy-Uniform-0}
		E_h ^z(t) \leq C(E_h^z(0) + \norm{g_h}^2_{L^1(0,T;L^2_h(0,1))}),
	\end{equation}
which implies \eqref{Energy-Uniform} providing a discrete Poincar\'e estimate proved hereafter:
	\begin{align*}
		\int_{(0,1)} |z_h|^2 & = h \sum_{j= 1}^N \left(h \sum_{k =0}^{j-1} \partial_h^+ (|z_h|^2)_k\right)
		\\
		& \leq 2 h \sum_{j=1}^N \left(\int_{[0,1)} |\partial_h^+ z_h|^2\right)^{1/2} \left(\int_{[0,1)} |m_h^+ z_h|^2\right)^{1/2}
		\\
		& \leq 2 \left(\int_{[0,1)}  |\partial_h^+ z_h|^2\right)^{1/2}  \left(\int_{(0,1)}  | z_h|^2\right)^{1/2},  
	\end{align*}
	which implies
	\begin{equation}
		\label{Poincare}
				\int_{(0,1)} |z_h|^2 \leq 4\int_{[0,1)}  |\partial_h^+ z_h|^2.
	\end{equation}
	Therefore \eqref{Energy-Uniform-0} implies \eqref{Energy-Uniform}.

	Finally, to prove \eqref{HiddenRegularityUniform}, we use a multiplier type argument. Multiplying the equation of $z_h$ by $j(z_{j+1,h}- z_{j-1,h})$ (which is a discrete version of $x\partial_x z$), summing in $j$ and integrating in time, we get (cf \cite[Lemma 2.2]{InfZua} or the proof of \eqref{MultiplierDiscrete} given hereafter in a more intricate case):
	\begin{multline}
		\label{MultiplierIdentity}
		h \sum_{j =0}^N\int_0^T \partial_t z_{j,h} \partial_t z_{j+1,h} + \int_0^T \int_{[0,1)} |\partial_h^+ z_h|^2 + X_h(t)\Big|_0^T \\
		- 2 \int_0^T \int_{(0,1)} g_h x \partial_h z_h = \int_0^T \left| (\partial_h^- z_h)_{N+1}\right|^2
	\end{multline}	
	where 
	$$
		X_h(t) = 2 \int_{(0,1)} x \partial_h z_h(t) \partial_t z_h(t).
	$$
	Of course, since each term in \eqref{MultiplierIdentity} is easily bounded by $\sup_{[0,T]}E_h^z(t)$ except for the term involving $g_h$ which can be bounded by $\norm{g_h}_{L^1(L^2_h)} \sup_{[0,T]}\sqrt{E_h^z(t)}$, we immediately obtain \eqref{HiddenRegularityUniform} from \eqref{Energy-Uniform}.
\end{proof}

\begin{proof}[Proof of Theorem \ref{TCWE3}.]

\noindent {\bf Step 1. Energy estimates.} 
Set $z_h=\partial_t u_h$. Then, using the notation $L_h[q_h] =\partial_{tt} - \Delta_h  + q_h ~$, $z_h$ satisfies
\begin{equation}
	\left\{\begin{array}{ll}
		\label{SDWE3}
			L_h[q_h]z_h =f_h \partial_t R_h, & \qquad t\in (0,T),\\
			z_{0,h}(t)=z_{N+1,h}(t)=0,& \qquad t\in (0,T),\\
			z_h(0)= 0, \quad\partial_t z_h(0)= f_h R_h(0). &
	\end{array}\right.
\end{equation}
We can apply Lemma \ref{LemEnergy} to $z_h$ solution of \eqref{SDWE3} since $\partial_t R_h$ belongs to 
$ L^1(0,T;L_h^{2}(0,1))$  and $f_h\in L_h^2(0,1)$.
In particular, if $E_h^z$ denotes the energy of $z_h$ (see \eqref{DefEnergy-H}), we obtain, for all $t \in (0,T)$,
\begin{align}
		E_h^z (t)  
		& \leq   C \left(  \norm{f_h R_h(0)}^2_{L^2_h(0,1)}  +  \norm{ f_h\partial_t R_h}^2_{L^1(0,T;L^2_h(0,1))}\right) 
		\nonumber\\
		& \leq C\norm{f_h}^2_{L^2_h(0,1)}  \left( \norm{R_h(0)}^2_{L_h^{\infty}(0,1)} + \norm{R_h}^2_{H^1(0,T;L_h^{\infty}(0,1))}\right) 
		 \leq C K^2 \norm{f_h}^2_{L^2_h(0,1)},
		\label{estimnrj}
\end{align}
where we have used that $H^1(0,T; L_h^\infty(0,1))$ embeds into $C([0,T]; L_h^\infty(0,1))$.

Moreover, Lemma \ref{LemEnergy} also yields \eqref{DirectEstimateUnif}. Indeed, estimate \eqref{HiddenRegularityUniform} becomes here
$$
	\norm{(\partial_h^- z_h)_{N+1} }^2_{L^2(0,T)} + \norm{ \partial_t z_h}_{L^\infty(0,T; L^2_h(0,1))}^2 \leq C K^2 \norm{f_h}^2_{L^2_h(0,1)},
$$
but $z_h=\partial_t u_h$ and the operator $h \partial_h^+ $ is bounded uniformly in $h$. \\

\noindent {\bf Step 2. The choice of the Carleman weight.} Since we assumed $T>1$, there exists $x^0 <0$ such that
$$
	T> \ds\sup_{x\in(0,1)}|x-x^0| \quad \Big(= 1 + |x^0|\Big).
$$
Therefore, we can choose $\beta\in(0,1)$ and $\eta > 0$ such that the Carleman weight function $\psi = \psi(t,x)=|x-x^0|^2-\beta t^2 +C_0$ satisfies
$$
	\left\{
		\begin{array}{l}
			\ds \psi(0,x)\geq   C_0, \quad x\in(0,1),
			\\
			\ds \psi(t,x) \leq C_0, \quad t \in [-T,-T+\eta] \cup [T-\eta,T],\  x \in (0,1).
		\end{array}
	\right.
$$
In particular,
\begin{equation}
	\label{phi}
		\left\{
			\begin{array}{l}
			\ds    \varphi(0,x) \geq e^{\lambda C_0}, \quad x \in (0,1),
		\\
			\ds  \varphi(t,x)  \leq e^{\lambda C_0} ,\quad t \in [-T,-T+\eta] \cup [T-\eta,T],\  x \in (0,1).
			\end{array}
		\right.
\end{equation}
In the sequel, we fix $\beta$ as above ($\beta \in (0,1)$ and $ T \sqrt{\beta}< \sup_{x\in(0,1)}|x-x^0|$), $\lambda$, $s_0$, $\varepsilon>0$ such that Corollary \ref{Corq} holds and the Carleman estimate \eqref{CarlemDq} holds for all $h \in (0,h_0)$ and $s \in (s_0, \varepsilon/h)$.\\

\noindent {\bf Step 3. Extension and truncation.}
We now extend the problem \eqref{SDWE3} on $(-T,T)$, setting $z_h(t)=z_h(-t)$ for all $t \in (-T, 0)$. We also extend $\partial_tR_h$ in an even way and keep the same notations for the new problem. 

Let us define the cut-off function $\chi\in C^{\infty}(\mathbb{R};[0,1])$ such that:
\begin{equation}
	\label{chi}
		\left\lbrace\begin{array}{ll}
			\chi(\pm T) = \partial_t \chi (\pm T) = 0
			\\
			\chi(t) =1 \quad \hbox{ for all } t\in [-T+\eta,T-\eta].
		\end{array} \right.
\end{equation}
We set $~w_h=\chi z_h~$ that satisfies the following equation:
\begin{equation}
	\left\{\begin{array}{ll}
		\label{SDWE4}
			L_h[q_h] w_h=\partial_{tt} \chi z_h + 2 \partial_t \chi \partial_t z_h + \chi f_h \partial_t R_h, \quad & t \in(-T,T),\\
			w_{0,h}(t)= w_{N+1,h}(t)=0,& t\in (-T,T),\\
			w_h(0)= 0, \quad\partial_t w_h(0)= f_h R_h(0), & \\
			w_h(\pm T)= 0, \quad\partial_t w_h(\pm T)= 0. &
	\end{array}\right.
\end{equation}

\noindent {\bf Step 4. Using the Carleman estimate.}
From now on, $C>0$ will correspond to a generic constant depending on $s_0, \lambda,T,x^0,\beta,\chi$ and $\eta$ 
but independent of $h \in (0,h_0)$ and $s\in(s_0, \varepsilon/h)$.
We use the same notations as in Section \ref{DCE} and set $v_h = \exp(s \varphi) w_h$. We then have (recall \eqref{P1})
$$	
	P_{h,1}v_h=\partial_{tt}v_h - (1+A_0) \Delta_h v_h +  s^2\la^2 \left[\varphi^2 \left( \partial_t \psi \right)^2 - A_2\right]v_h
$$
and $v_h(\pm T)  =  \partial_t v_h(\pm T) = 0 $, $ v_{h}(0)  =  0$ and $\partial_t v_h(0) = f_h R_h(0) e^{ s \varphi(0, \cdot)}$.

Using the properties of $v_h$, Lemma \ref{AA} and \eqref{A0}, we can make the following calculation:
\begin{align*}
	\lefteqn{
	\int_{-T}^0\int_{(0,1)} P_{h,1}v_h~ \partial_t v_h 
	}
	\\
	 = &
	  \int_{-T}^0\int_{(0,1)} \left(\partial_{tt}v_h - (1+A_0) \Delta_h v_h  +  s^2\la^2 \left[\varphi^2 (\partial_t \psi)^2 - A_2\right]v_h\right) \partial_t v_h 
	 \\
	= & 
	~\dfrac 12\int_{(0,1)} |\partial_t v_h(0)|^2 
	+ \int_{-T}^0 \int_{[0,1)} \partial_h^+ v_h \partial_h^+ ((1+A_0) \partial_t v_h) 
	\\
	&
	\qquad - \dfrac {s^2\lambda^2}2  \int_{-T}^0\int_{(0,1)}   |v_h|^2\partial_t\left(\varphi^2 \left( \partial_t \psi \right)^2 - A_2\right)
	\\
	 \geq & 
	 ~\dfrac 12\int_{(0,1)} | f_h|^2 |R_h(0)|^2 e^{2 s \varphi(0, \cdot)} 
	- s^2 C  \int_{-T}^0\int_{(0,1)}  |v_h|^2
	\\
	& 
	\qquad
	+ \int_{-T}^0 \int_{[0,1)}\left(\frac{1}{2} \partial_t(|\partial_h^+ v_h|^2) m_h^+ (1+A_0) + \partial_h^+ A_0 \partial_h^+ v_h m_h^+ \partial_t v_h\right)
	\\
	 \geq & 
	 ~\frac{r^2}{2}\int_{(0,1)} |f_h|^2 e^{2 s \varphi(0, \cdot)} 
	- s^2 C  \int_{-T}^0\int_{(0,1)}  |v|^2
	\\
	& 
	\qquad -\frac{1}{2} \int_{-T}^0 \int_{[0,1)} |\partial_h^+ v_h|^2 m_h^+ (\partial_t A_0)+ \int_{-T}^0 \int_{[0,1)}\partial_h^+ A_0 \partial_h^+ v_h m_h^+ \partial_t v_h
	\\
	\geq & 
	~ \frac{r^2}{2}\int_{(0,1)} |f_h|^2 e^{2 s \varphi(0, \cdot)} 
	- s^2 C  \int_{-T}^0\int_{(0,1)}  |v|^2
	\\
	& 
	\qquad- \mathcal{O}_\lambda(sh) \left(  \int_{-T}^0 \int_{[0,1)} |\partial_h^+ v_h|^2 +  \int_{-T}^0 \int_{(0,1)} |\partial_t v_h|^2\right).
\end{align*}
Therefore
\begin{eqnarray*}
	 \frac{r^2}{2}\int_{(0,1)} |f_h|^2 e^{2 s \varphi(0, \cdot)} 
		 & \leq & 
	\int_{-T}^T\int_{(0,1)} P_{h,1}v_h~ \partial_t v_h 
	+
	C s^2\int_{-T}^T\int_{(0,1)}  |v_h|^2 
	\\
		& & 
	+ \mathcal{O}_\lambda(sh) \left(  \int_{-T}^T \int_{[0,1)} |\partial_h^+ v_h|^2 +  \int_{-T}^T \int_{(0,1)} |\partial_t v_h|^2\right).
\end{eqnarray*}
Using 
$$
	\left|\int_{-T}^T\int_{(0,1)} P_{h,1}v_h~ \partial_t v_h \right| 
	\leq \frac{1}{2\sqrt{s}} \left(  \int_{-T}^T\int_{(0,1)} |P_{h,1}v_h|^2 \right.
	+\left. s \int_{-T}^T\int_{(0,1)} |\partial_t v_h|^2\right)
$$
and the fact that $\mathcal{O}_\lambda(sh)$ is bounded by some constant independent of $s$ since $sh \leq \varepsilon$, we get
\begin{multline*}	
	\lefteqn{	
	 r^2\sqrt{s} \int_{(0,1)} |f_h|^2 e^{2 s \varphi(0, \cdot)} 
		\leq
	 \int_{-T}^T\int_{(0,1)} |P_{h,1}v_h|^2+ s \int_{-T}^T\int_{(0,1)} |\partial_t v_h|^2
	}\\
	+ C s^{5/2} \int_{-T}^T\int_{(0,1)}  |v_h|^2  
	+ C\sqrt{s} \left(  \int_{-T}^T \int_{[0,1)} |\partial_h^+ v_h|^2 +  \int_{-T}^T \int_{(0,1)} |\partial_t v_h|^2\right).
\end{multline*}
From the Carleman estimate \eqref{decompo} of Lemma~\ref{LemDecompo}, this implies that for all $s$ satisfying $s_0<s<\dfrac \varepsilon h$, 
\begin{eqnarray}
	r^2 \sqrt{s} \int_{(0,1)} |f_h|^2 e^{2 s\varphi(0, \cdot)} 
		&\leq &
	M\int_{-T}^{T} \int_{(0,1)} |P_h v_h |^2 	+M s \int_{-T}^{T}  \left|(\partial_h^-v_h)_{N+1} \right|^2 
	\nonumber\\ 
		&&
	+ M s \int_{-T}^{T} \int_{[0,1)} |h \partial_h^+\partial_t v_h|^2 
	\nonumber\\
		& \leq &
	M \int_{-T}^{T} \int_{(0,1)} e^{2s\varphi} |L_h w_h |^2  
	+
	M s \int_{-T}^{T} e^{2s\varphi(t,1)} \left| (\partial_h^-w_h)_{N+1} \right|^2 
	\nonumber\\
		& &
	+ M s \int_{-T}^{T} \int_{[0,1)}e^{2s\varphi} |h \partial_h^+\partial_t w_h|^2,
	\label{Carestim} 
\end{eqnarray}
where the last estimate follows from \eqref{BoundaryTerms}--\eqref{TycV-W}.\\

From equation \eqref{SDWE4}, the properties \eqref{chi} of the cut-off function $\chi$ and \eqref{phi} of the weight function $\varphi$, which is a decaying function of $|t|$, one gets
\begin{eqnarray*}
	\lefteqn{
	\int_{-T}^{T} \int_{(0,1)} e^{2s\varphi} |L_h w_h |^2  
	 \leq 
	C \int_{-T}^T\int_{(0,1)} e^{2s\varphi} \left( |\chi f_h \partial_t R_h|^2 + |\partial_t \chi \partial_t z_h|^2 + |\partial_{tt} \chi z_h|^2\right) 
	}\\
		 & \leq &  
	C \int_{-T}^T\int_{(0,1)} e^{2s\varphi}| f_h|^2 | \partial_t R_h|^2 
	+ C \left( \int_{-T}^{-T+\eta} + \int_{T-\eta}^T \right)\int_{(0,1)} e^{2s\varphi}\left( |\partial_t z_h|^2 + |z_h|^2\right)
	\\
		& \leq &  
	C K^2 \int_{(0,1)} e^{2s\varphi(0, \cdot)}| f_h|^2
	+ C e^{2s e^{\lambda C_0}} \left( \int_{-T}^{-T+\eta} + \int_{T-\eta}^T \right)E^z_h (t).
	\end{eqnarray*}
Using now the energy estimate \eqref{estimnrj},
\begin{eqnarray}
	\int_{-T}^{T} \int_{(0,1)} e^{2s\varphi} |L_hw_h |^2  
	& \leq &
	C K^2 \int_{(0,1)} e^{2s\varphi(0, \cdot)}| f_h|^2 +C K^2 e^{2s e^{\lambda C_0} }  \int_{(0,1)} | f_h|^2
	\nonumber
	\\
		&
		\leq &
	C K^2 \int_{(0,1)} e^{2s\varphi(0, \cdot)}| f_h|^2. 
	\label{EstimeesLW}
\end{eqnarray}

Similarly, since $\partial_t w_h = \chi \partial_t z_h + \partial_t \chi z_h$, using the energy estimate \eqref{estimnrj},
\begin{eqnarray}
	\lefteqn{\int_{-T}^{T} \int_{[0,1)}e^{2s\varphi} |h \partial_h^+\partial_t w_h|^2}
	\nonumber\\
		& \leq &
	2 \int_{-T}^{T} \int_{[0,1)}e^{2s\varphi} \chi^2 |h \partial_h^+\partial_t z_h|^2
		+
	2 h^2 \int_{-T}^T \int_{[0,1)} e^{2 s \varphi} |\partial_t \chi|^2 |\partial_h^+ z_h|^2
	\nonumber\\
		& \leq &
	2 \int_{-T}^{T} \int_{[0,1)}e^{2s\varphi} \chi^2 |h \partial_h^+\partial_t z_h|^2
		+
	2 h^2 CK^2  \int_{(0,1)} e^{2s\varphi(0, \cdot)}| f_h|^2, 
	\label{EstTycTerm}
\end{eqnarray}
where we have used that, as proved above,
$$
	\int_{-T}^T \int_{[0,1)} e^{2 s \varphi} |\partial_t \chi|^2 |\partial_h^+ z_h|^2 \leq C e^{2s e^{\lambda C_0}} \left( \int_{-T}^{-T+\eta} + \int_{T-\eta}^T \right)E^z_h (t) \leq 	C K^2 \int_{(0,1)} e^{2s\varphi(0, \cdot)}| f_h|^2.
$$

Therefore, plugging \eqref{EstimeesLW}--\eqref{EstTycTerm} in \eqref{Carestim} we obtain
\begin{multline*}
	\sqrt{s} r^2 \int_{(0,1)} e^{2s\varphi(0, \cdot)}|f_h|^2 
		\leq
	 C K^2 \int_{(0,1)} e^{2s\varphi(0, \cdot)}| f_h|^2  
	+C s \int_{-T}^{T} e^{2s\varphi(t,1)}\chi^2 \left|( \partial_h^-z_h)_{N+1} \right|^2	
		\\
	+ C s \int_{-T}^{T} \int_{[0,1)} e^{2s\varphi}|h \partial_h^+\partial_t z_h|^2 + C s h^2  K^2  \int_{(0,1)} e^{2s\varphi(0, \cdot)}| f_h|^2.
\end{multline*}
Thus, since $sh^2 \leq \varepsilon(h) h_0\leq 1$, taking $s_*>s_0$ such that for all $s\geq s_*$, $\sqrt{s} r^2  - 2C K^2 >0$, for all $h \in(0,h_*)$ with $h_* = \min\{ h_0, \varepsilon/s_*\}$, we obtain
$$
		\int_{(0,1)}e^{2s_*\varphi(0, \cdot)} |f_h|^2 
	\leq C s_* \int_{-T}^{T} e^{2s_*\varphi(t,1)}\chi^2 \left|( \partial_h^-z_h)_{N+1} \right|^2	
	+ C s_* \int_{-T}^{T} \int_{[0,1)} e^{2s_*\varphi}|h \partial_h^+\partial_t z_h|^2,
$$
and therefore
\begin{equation}
	\label{StabilityEstimateInZ}
	\norm{f_h}_{L^2_h(0,1)} \leq C \norm{(\partial_h^-z_h)_{N+1}}_{L^2(-T,T)} + C \norm{h \partial_h^+\partial_t z_h}_{L^2(-T,T;L^2_h(0,1))},
\end{equation}
which coincides with \eqref{ReverseEstimateUniform}. The proof of Theorem~\ref{TCWE3} is then complete.
\end{proof}

\begin{remark}\label{RemarkFiltered}
	With the notations of Theorem \ref{TCWE3}, if $R_h \in H^1(0,T; L^\infty_h(0,1))\cap W^{2,1} (0,T; L^2_h(0,1))$ and $R_h(0,\cdot), \partial_t R_h(0, \cdot) \in L^\infty_{h} (0,1)$, considering the equation satisfied by $w_h = \partial_t z_h$:
	$$
	\left\{\begin{array}{ll}
		\partial_{tt} w_h - \Delta_h w_h+ q_h w_h = f_h \partial_{tt}R_h, \quad 
			&	 t\in (0,T), \, j\in \llbracket 1,N \rrbracket,\\
			w_{0,h}(t)= w_{N+1,h}(t)=0,
			&  t\in (0,T),\\
			w_h(0)= f_h R_h(0), \quad\partial_t w_h(0)= f_h \partial_{t}R_h(0), &
	\end{array}\right.
	$$
	from Lemma \ref{LemEnergy}, we get
	$$
		\sup_{t \in (0,T)} E_h^w (t) \leq C E_h^w (0) + C \norm{f_h}_{L^2_h(0,1)}^2 \norm{R_h}_{ W^{2,1} (0,T; L^2_h(0,1))}^2, 
	$$
	with
	\begin{align*}
		E_h^w (0) 
		&= \int_{[0,1)} |\partial_h^+( f_h R_h(0))|^2 + \int_{(0,1)} |f_h \partial_t R_h(0)|^2 + \int_{(0,1)} |f_h R_h(0)|^2 
		\\
		& \leq 2 \int_{[0,1)} |\partial_h^+( f_h)|^2 |m_h^+ R_h(0)|^2+ C \norm{f_h}_{L^2_h(0,1)}^2\norm{(R_h(0,\cdot), \partial_t R_h(0, \cdot))}_{L^\infty_h(0,1)^2}^2. 
	\end{align*}
	Therefore, if in addition to \eqref{DiscreteSource}, there exists a constant $\delta >0$  such that for all $h>0$,
	\begin{equation}
		\label{FilteringCondition}
		\begin{array}{l}
			\ds \norm{R_h}_{ W^{2,1} (0,T; L^2_h(0,1))}+ \norm{(R_h(0,\cdot), \partial_t R_h(0, \cdot))}_{L^\infty_h(0,1)^2} \leq \frac{\delta}{h}, 
		\\
		\ds 	 \int_{[0,1)} |\partial_h^+ f_h|^2  \leq \frac{\delta^2}{h^{2}} \int_{(0,1)} |f_h|^2,
		\end{array}
	\end{equation}
	then
	$$
		\norm{\partial_h^+ w_h }_{L^\infty(0,T;L^2_h[0,1))}  = \norm{\partial_h^+ \partial_t z_h }_{L^\infty(0,T;L^2_h[0,1))} \leq C \delta h^{- 1} \norm{f_h}_{L^2_h(0,1)},
	$$
	and in particular,
	$$
		h\norm{\partial_h^+ \partial_t z_h }_{L^\infty(0,T;L^2_h[0,1))} \leq C \delta \norm{f_h}_{L^2_h(0,1)}.
	$$
	Therefore, if condition \eqref{FilteringCondition} is satisfied for $\delta>0$ small enough, estimate \eqref{StabilityEstimateInZ} simply becomes, for $h$ small enough,
	$$
		\norm{f_h}_{L^2_h(0,1)} \leq C \norm{(\partial_h^-z_h)_{N+1}}_{L^2(-T,T)}.
	$$
	Condition \eqref{FilteringCondition} can be seen as a filtering condition on the data. To be more precise, if we filter enough the data (at the scale $\delta/h$ with  $\delta$ small enough ), the Tychonoff regularization term is not needed anymore in \eqref{ReverseEstimateUniform}.
\end{remark}

\subsection{Uniform stability for the discrete inverse problem}\label{ProofThmTCWE2}
\begin{proof}[Proof of Theorem \ref{TCWE2}]
	Setting $ u_h = y_h[q_h] - y_h[p_h]$, where $y_h[q_h]$ and $y_h[p_h]$ are respectively  the solutions of \eqref{SDWE1} corresponding to $p_h$ and $q_h$, then $u_h$ solves
\begin{equation}
	\left\{\begin{array}{ll}\label{SDWE5}
		\partial_{tt}  u_h - \Delta_h u_h+ q_h u_h = f_h R_h, \quad 
			&	 t\in (0,T), \, j\in \llbracket 1,N \rrbracket,\\
		u_{0,h}(t)= u_{N+1,h}(t)=0,
			&  t\in (0,T),\\
		u_h(0)= 0, \quad\partial_t u_h(0)= 0, &
	\end{array}\right.
\end{equation}
	with $f_h = p_h - q_h$ and $R_h = y_h[p_h]$. We then directly apply Theorem \ref{TCWE3}.
\end{proof}

\begin{remark}\label{RemFiltering}
	Remark \ref{RemarkFiltered} also applies here of course, and the filtering condition \eqref{FilteringCondition} then becomes: 
	$$
		\norm{y_h[p_h]}_{W^{2,1}(0,T;L^2_h(0,1))} + \norm{(y^0_h, y^1_h)}_{L^\infty_h(0,1)^2} \leq \frac{\delta}{h},	\qquad 		 \int_{[0,1)} \left|\partial_h^+ (q_h- p_h)\right|^2  \leq \frac{\delta^2}{h^{2}} \int_{(0,1)} |q_h-p_h|^2,
	$$
	for some $\delta>0$ small enough.
	A convenient way to satisfy these two conditions is to impose that both $p_h$ and $q_h$ belong to a filtered space and that the data $(y_{h}^0, y_h^1)$, $g_h$ and $(g_h^0, g_h^1)$ are smooth.
\end{remark}

\section{Convergence issues}\label{Cv}
In this section, we will detail and prove the convergence results that were presented rapidly in the introduction.
\subsection{Statements of the results}\label{SubsecCv}
In order to prove a convergence result, we shall need some assumptions first.

\begin{assumption}[A priori bounds on the potential]
	\label{AssBounds}
	There exists $m>0$ such that $p \in L^\infty_{\leq m} (0,1)$.
\end{assumption}

\begin{assumption}[Regularity assumptions]
	\label{AssRegContinue}	
	The data satisfy
	\begin{align*}
		& (y^0,y^1)  \in   H^2(0,1) \times H^1(0,1),
		\\
		& g \in   W^{1,1}(0,T;L^2(0,1)),
		\quad
		(g^0, g^1) \in (H^2(0,T))^2,
	\end{align*}
	with the compatibility conditions
	$$
		 g^0 (0) = y^0(0), \quad g^1(0) = y^0(1), \quad \partial_t g^0(0) = y^1(0) \hbox{  and  } \partial_t g^1(0) = y^1(1).
	$$
\end{assumption}
One should notice that under these regularity assumptions, according to \cite{LasieckaLionsTriggiani} (see also Remark~\ref{RemReg}), for $p \in L^\infty_{\leq m} (0,1)$ the solution $y[p]$ of \eqref{CWE1} belongs to the space $C^2([0,T]; L^2(0,1)) \cap C^1([0,T]; H^1(0,1)) \cap C^0([0,T]; H^2(0,1))$. In particular, one can check that $ \partial_{tx} y[p](\cdot, 1) \in L^2(0,T)$ (this can be found in \cite{LasieckaLionsTriggiani} but can also be seen as a consequence of the multiplier identity \eqref{multid}) and that $ y[p] \in H^1(0,T; L^\infty (0,1))$.\\

Since we are interested in a convergence result, we shall explain how to compare discrete functions with continuous ones. In order to do so, we introduce two extension operators.

The first one extends discrete functions by continuous piecewise affine functions. To be more precise, if $f_h$ is a discrete function $(f_{j,h})_{j \in \llbracket 0, \cdots, N+1\rrbracket }$, the extension $\eh (f_h)$ defined on $[0,1]$ by 
$$
	\eh(f_h)(x) = f_{j,h} + \left( \frac{f_{j+1,h}- f_{j,h}}{h}\right)( x-jh)	\quad \hbox{ on } [jh, (j+1)h], \ j \in \llbracket 0, \cdots, N\rrbracket.
$$
This extension presents the advantage of being naturally in $H^1(0,1)$.

The second one is the piecewise constant extension $\eh^0(f_h)$, defined for discrete functions $(f_{j,h})_{j \in \llbracket 1, \cdots, N\rrbracket }$ by
$$
	\begin{array}{ll}
	\eh^0 (f_h) = f_{j,h}  \quad &\hbox{ on } [(j-1/2)h, (j+1/2)h[, \, j \in \llbracket 1, \cdots, N \rrbracket,
	\smallskip\\
	\eh^0 (f_h) = 0  \quad &\hbox{ on } [0, h/2[\cup [ (N+1/2)h,1].
	\end{array}
$$
Of course, this one is more natural when dealing with functions lying in $L^2(0,1)$. In particular, we have
$$
	\norm{ \eh^0 (f_h)}_{L^2(0,1)} = \norm{ f_h}_{L^2_h(0,1)}.
$$

Also note that easy (but tedious) computations show that $\eh(f_h)$ converge to $f$ in $L^2(0,1)$ if and only if $\eh^0(f_h)$ converge to $f$ in $L^2(0,1)$.\\

Of course, we shall need some convergence estimates:

\begin{assumption}[Convergence assumptions]
	\label{AssConvergence}
	The sequence of discrete data $(y^0_h, y^1_h)$ satisfies
	\begin{equation}
		\label{ConvInitData}
		(\eh^0 (\Delta_h y^0_h), \eh^0 (\Delta_h y^1_h) )\underset{h \to 0} \longrightarrow (\Delta y^0, \Delta y^1) \quad \hbox{ in } L^2(0,1) \times H^{-1}(0,1).
	\end{equation}
	The sequences of source terms $g_h, (g^0_h, g^1_h)$ satisfy 
	\begin{equation}
		\label{ConvSourceTerms}
		\eh^0 (g_h) \underset{h \to 0} \longrightarrow g \quad \hbox{ in } W^{1,1}(0,T; L^2(0,1)), \quad (g^0_h, g^1_h) \underset{h \to 0} \longrightarrow (g^0, g^1) \quad \hbox{ in } (H^2(0,T))^2.
	\end{equation}
\end{assumption}
Finally, we shall also need a uniform positivity assumption:
\begin{assumption}[Positivity]\label{AssPositivity}
	There exists $r>0$ such that 	
	\begin{equation}
		\label{EqPositivityConv}
		\inf \left\{|y^0(x)|,x\in (0,1)\right\} \geq r> 0, 
			\quad \hbox{and} \quad 
		\forall h >0, \quad \inf_{ j \in \llbracket 1, N\rrbracket} |y^{0}_{j,h}|\geq r.
	\end{equation}
\end{assumption}

Now, we  introduce, for $h>0$, the following observation operator:

\begin{equation}
	\label{DiscreteFh}
	\begin{array}{lcll}
		\Theta_h :  	&L^\infty_{h,\leq m} (0,1) &\to & L^2(0,T)\times L^2((0,T)\times(0,1)) \\
		&p_h &\mapsto &\big( \partial_t (\partial_h^-y_h[p_h])_{N+1} , h\partial_x \eh( \partial_{tt} y_h[p_h]) \big),
	\end{array}
\end{equation}
where $y_h[p_h]$ is the solution of \eqref{SDWE1} with potential $p_h$.
We also introduce its continuous analogous 
\begin{equation}
	\begin{array}{lcll}
		\Theta_0 : 	&L^\infty_{\leq m} (0,1) &\to&  L^2(0,T)\times L^2((0,T)\times(0,1)) \\
		& p &\mapsto&\big( \partial_t \partial_x y[p](\cdot,1), 0 \big), 
	\end{array}
\end{equation}
where $y[p]$ is the solution of \eqref{CWE1} with potential $p$.

Note that, using these notations, Theorem \ref{TCWE2} can then be seen as a uniform stability of the maps $\Theta_h^{-1}$. Indeed, 
\begin{multline*}
	\norm{h \partial_h^+\partial_{tt} y_h[q_h] - h \partial_h^+\partial_{tt} y_h[p_h]}_{L^2(0,T;L^2_h([0,1))} 
	\\
	=
	 \norm{h \partial_x \eh( \partial_{tt} y_h[q_h])-h \partial_x \eh( \partial_{tt} y_h[p_h]) }_{L^2((0,T)\times (0,1))},
\end{multline*}
and then \eqref{UniformStability-3} reads as:
\begin{equation}
	\label{UniformStability-4}
	\norm{\eh^0 (q_h) - \eh^0(p_h)}_{L^2(0,1)} \leq C  \norm{ \Theta_h(p_h)  - \Theta_h(q_h) }_{L^2(0,T) \times L^2((0,T)\times (0,1))}.
\end{equation}

Our main result is then the following convergence theorem:

\begin{theorem}\label{ThmCV}
Under Assumptions \ref{AssBounds}--\ref{AssPositivity}, let $q_h \in L^\infty_{h,\leq m}(0,1)$ be such that 
\begin{equation}
	\label{ConvergenceObs}
	\Theta_h(q_h) \underset{h \to 0}\longrightarrow \Theta_0(p) \quad \textit{strongly in } L^2(0,T)\times L^2((0,T)\times(0,1)).
\end{equation}
Then one has the convergence 
\begin{equation}
	\label{ConvergencePotentials}
	\eh^0 (q_h) \underset{h \to 0}\longrightarrow p \quad \textit{ in } L^2(0,1).
\end{equation}
\end{theorem}
Before going into the proof of Theorem \ref{ThmCV}, we shall emphasize that there exist discrete sequences of potentials 
such that \eqref{ConvergenceObs} holds. Actually, this is a consequence of the following consistency result:

\begin{theorem}\label{ThmConsistence}
	Under Assumptions \ref{AssBounds}--\ref{AssPositivity}, for all potential $p \in L^\infty_{\leq m}(0,1)$ there exists discrete potentials $p_h \in L^\infty_{h,\leq m}(0,1)$ such that
	\begin{equation} 
	\label{ConsistanceFh}
		\eh^0 ( p_h )\underset{h \to 0}\longrightarrow p  \quad \hbox{ in } L^2(0,1) \quad
		\hbox{ and }
		\quad \Theta_h(p_h) \underset{h \to 0}\longrightarrow \Theta_0(p) \quad \hbox{ in } L^2(0,T)\times L^2((0,T)\times(0,1)).
	\end{equation}
Moreover
	\begin{equation}
		\label{BoundsOnDtYh}
		\sup_{h \in (0,1)} \|y_h[p_h] \|_{H^1(0,T; L^\infty_h(0,1))} < \infty,
	\end{equation}
	where $y_h[p_h]$ is the solution of \eqref{SDWE1}.
\end{theorem}

In the following section, we give the proofs of these Theorems. Actually, as we will see, Theorem~\ref{ThmConsistence} is the second milestone of the proof of Theorem~\ref{ThmCV}, the first one being Theorem~\ref{TCWE2}.
In other words, the proof of Theorem~\ref{ThmCV}, that will be given at the end of this section, relies on a Lax-type argument for the convergence of the numerical schemes based on the consistency of the method, given by \eqref{ConsistanceFh}, and the uniform stability 
\eqref{UniformStability-4} of the discrete inverse problems.

\subsection{Proofs}
\begin{proof}[Proof of Theorem \ref{ThmConsistence}]

In the proof, we shall distinguish the regularity and convergence issues coming from the boundary source terms and the initial data from the classical ones coming from the potential and distributed source term.\\

\noindent {\bf Step 1: Convergence without potential and source term.}\\
Let $z$ be the solution of 
\begin{equation}
	\label{CWE1z}
		\left\{\begin{array}{ll}
			\partial_{tt}z - \partial_{xx}z=0,  & (t,x) \in  (0,T)\times (0,1),
				\\
			z(t,0)=g^0(t),\quad y(t,1)=g^1(t),&  t\in (0,T),
				\\
			z(0,\cdot)= y^0, \quad \partial_t z(0,\cdot)= y^1. &
		\end{array}
		\right.
\end{equation}
Since $(y^0, y^1)$ and $(g^0, g^1)$ satisfy Assumption \ref{AssRegContinue}, the solution $z$ of \eqref{CWE1z} satisfies (see \cite{LasieckaLionsTriggiani} for details)
$$
	\partial_t z \in C([0,T]; H^1(0,1))\cap C^1([0,T], L^2(0,1)).
$$ 
We are therefore allowed to write the following multiplier identity
\begin{multline}\label{multid}
	\frac{1}{2} \int_0^T  |\partial_{xt} z(t,1)|^2 \, dt
	= \frac 12 \int_0^T\int_0^1 \left( |\partial_{tt} z|^2 +  |\partial_{xt} z|^2  \right)\, dxdt
	- \frac 12 \int_0^T  |\partial_t g^1|^2 \, dt
	\\
	+ \int_0^1 \partial_{tt} z(T,x)x \partial_{xt} z(T,x) \, dx 
	- \int_0^1  \partial_{tt} z(0,x) x \partial_{xt} z (0, x) \, dx,
\end{multline}
which is obtained by differentiating in time equation \eqref{CWE1z}, and then multiplying it by $x \partial_{tx} z$, integrating over $(0,T)\times (0,1)$ and doing integration by parts.

Now, let $z_h$ be the solution of 
\begin{equation}
	\label{SDWE1z}
	\left\{\begin{array}{ll}
		\partial_{tt}z_{j,h} - \left(\Delta_h z_{h}\right)_j=0, \quad & t\in (0,T), \ j\in \llbracket 1,N \rrbracket,\\
		z_{0,h}(t)=g_h^0(t),\quad z_{N+1,h}(t)=g_h^1(t),& t\in [0,T],\\
		z_{j,h}(0)= y_{j,h}^0, \quad\partial_t z_{j,h}(0)= y_{j,h}^1, &  j\in \llbracket 1,N \rrbracket.
	\end{array}\right.
\end{equation}
In this step, we want to prove that $z_h \underset{h \to 0}\longrightarrow z$ in the appropriate functional spaces.
In order to do that, we use the following result:

\begin{theorem}[\cite{ErvZuaBook}]\label{ThmConvTransposition}
Let $(f_h^0, f_h^1)_{h>0}$ be a sequence of boundary data strongly convergent to some functions $(f^0, f^1)$ in $L^2(0,T)^2$. Let $(\varphi_h^0, \varphi_h^1)$ be a sequence of discrete functions such that 
\begin{equation}\label{12ConvInitStrong}
	(\eh^0 (\varphi^0_h), \eh^0 (\varphi_h^1)) \underset{h \to 0}\longrightarrow (\varphi^0, \varphi^1) \quad \hbox{strongly in} \quad    L^2(0,1) \times H^{-1}(0,1).  
\end{equation}

Then the solutions $\varphi_h$ of 
$$
	\left\{\begin{array}{ll}
		\partial_{tt} \varphi_{j,h} - \left(\Delta_h  \varphi_{h}\right)_j=0, \quad & t\in (0,T), \ j\in \llbracket 1,N \rrbracket,\\
		 \varphi_{0,h}(t)=f_h^0(t),\quad  \varphi_{N+1,h}(t)=f_h^1(t),& t\in [0,T],\\
		 \varphi_{j,h}(0)=  \varphi_{j,h}^0, \quad\partial_t  \varphi_{j,h}(0)=  \varphi_{j,h}^1, &  j\in \llbracket 1,N \rrbracket
	\end{array}\right.
$$
converge toward the solution $\varphi$ of
$$
		\left\{\begin{array}{ll}
			\partial_{tt} \varphi - \partial_{xx} \varphi=0, \quad & (t,x) \in  (0,T)\times (0,1),
				\\
			\varphi(t,0)=f^0(t),\quad \varphi(t,1)=f^1(t),&  t\in (0,T),
				\\
			\varphi(0,\cdot)= \varphi^0, \quad \partial_t \varphi(0,\cdot)= \varphi^1, &
		\end{array}
		\right.
$$
in the following sense: for all $p<\infty$,
\begin{equation}\label{ConvergenceY12Strong} 
	\eh^0 (\varphi_h) \underset{h \to 0}{\longrightarrow} \varphi \quad \hbox{strongly in } L^p((0,T); L^2 (0,1))\cap W^{1,p} ((0,T);H^{-1}(0,1)) .
\end{equation}
Besides, for all $t_0\in [0,T]$, 
\begin{equation}\label{ConvergenceEnT=t}
	(\eh^0 (\varphi_h)(t_0), \partial_t \eh^0 (\varphi_h)(t_0)) \underset{h \to 0}{\longrightarrow} (\varphi(t_0), \partial_t \varphi (t_0)) \quad \hbox{strongly in } L^2(0,1) \times H^{-1}(0,1). 
\end{equation}
\end{theorem}

For the proof of Theorem \ref{ThmConvTransposition}, we refer to \cite{ErvZuaBook}. Note that Theorem \ref{ThmConvTransposition} is not standard, since it deals with solutions of the continuous wave equation defined in the transposition sense. Therefore, the proof of Theorem \ref{ThmConvTransposition} is based on a duality argument and convergence results for the adjoint equation, namely the waves, and in particular on their normal derivatives on the boundary (which corresponds to the adjoint operator of the Dirichlet boundary conditions).\\

Of course, regarding the regularity hypothesis in Assumption \ref{AssRegContinue} and the convergence one in Assumption \ref{AssConvergence}, we can apply this result to $z_h$, $\partial_t z_h$ and $\partial_{tt} z_h$. Of course, the latter yields the strongest result, thus improving the ones on $\partial_t z_h$ and $z_h$: for all $p<\infty$,
\begin{equation}
	\label{ConvergenceZ_hToZ}
	\left\{
	\begin{array}{cccl}
	\partial_{tt} \eh^0 (z_h) &\underset{h \to 0}{\longrightarrow}& \partial_{tt} z& \ \hbox{strongly in } L^p((0,T); L^2 (0,1))\cap W^{1,p} ((0,T);H^{-1}(0,1)), 
	\\
	\partial_{t} \eh (z_h) &\underset{h \to 0}{\longrightarrow}& \partial_{t} z& \ \hbox{strongly in } L^p((0,T); H^1 (0,1))\cap W^{1,p} ((0,T);L^2(0,1)),
	\\
	 \eh( z_h) &\underset{h \to 0}{\longrightarrow}&  z &\ \hbox{strongly in } W^{1,p}((0,T); H^1 (0,1))\cap W^{2,p} ((0,T);L^2(0,1)),
	\end{array}
	\right.
\end{equation}
and, for all $t_0 \in [0,T]$, 
\begin{equation}
	\label{ConvergenceZ_hToZ-t0}
	(\eh^0 (\partial_h \partial_t z_h)(t_0), \eh^0 (\partial_{tt} z_h)(t_0))  \underset{h \to 0}{\longrightarrow} (\partial_{xt} z(t_0), \partial_{tt} z(t_0))  \quad \hbox{strongly in } (L^2(0,1))^2.
\end{equation}

Now, we focus on the convergence of the normal derivatives. This is slightly more subtle. First, arguing as in \cite{ErvZuaBook} by duality against smooth functions, one easily checks that 
\begin{equation}
	\label{WeakConvergence}
	(\partial_h^- z_h)_{N+1}  \underset{h \to 0}{\rightharpoonup}  \partial_x z(\cdot, 1) \quad \hbox{weakly in } L^2(0,T).
\end{equation}
Then, we derive a multiplier identity similar to \eqref{multid} for the discrete equation \eqref{SDWE1z}. In order to do that, we multiply  equation \eqref{SDWE1z} differentiated once in time by $x\partial_h \partial_t z_h$:
$$
	 \int_0^T \int_{(0,1)} \partial_{ttt} z_h \ x\  \partial_h \partial_t z_h\, dt - \int_0^T \int_{(0,1)}\Delta_h \partial_t z_h\,  x \ \partial_h \partial_t z_h \, dt = 0.
$$
But, on the one hand, using \eqref{IPP0}, we have
\begin{align*}
	\lefteqn{
	 \int_0^T \int_{(0,1)} \partial_{ttt} z_h\  x \ \partial_h \partial_t z_h \, dt 
		=
	\left.\int_{(0,1)} \partial_{tt} z_h\  x\  \partial_h \partial_t z_h\right|_0^T
	- \int_0^T \int_{(0,1)} \partial_{tt} z_h \ x \ \partial_h \partial_{tt} z_h  \, dt
	}
	 \\
	 &	=
	 \left.\int_{(0,1)} \partial_{tt} z_h\  x\  \partial_h \partial_t z_h\right|_0^T
	+\frac{1}{2} \int_0^T \int_{(0,1)} |\partial_{tt} z_h|^2 \, dt
	 - \frac{h^2}{4} \int_0^T \int_{[0,1)} |\partial_h ^+ \partial_{tt} z_h|^2 \, dt
\end{align*}
and on the other hand, using now \eqref{IPP2}, we get
$$
	 \int_{(0,1)}\Delta_h \partial_t z_h\,  x \ \partial_h \partial_t z_h
	 	=
	-\frac{1}{2} \int_{[0,1)} |\partial_h^+ \partial_t z_h|^2 + \frac{1}{2} |(\partial_h^- \partial_t z_h)_{N+1}|^2.
$$
Combining these last three identities, we obtain
\begin{multline}
	\label{MultiplierDiscrete}
	\frac{1}{2} \int_0^T |(\partial_h^- \partial_t z_h)_{N+1}|^2 \, dt
	+ 
	\frac{h^2}{4} \int_0^T \int_{[0,1)} |\partial_h ^+ \partial_{tt} z_h|^2\, dt
	\\
	= 
	\left.\int_{(0,1)} \partial_{tt} z_h\  x\  \partial_h \partial_t z_h\right|_0^T
	+\frac{1}{2} \int_0^T \int_{(0,1)} |\partial_{tt} z_h| ^2\, dt + \frac{1}{2} \int_0^T \int_{[0,1)} |\partial_h^+ \partial_t z_h|^2 \, dt.
\end{multline}

According to the strong convergences in \eqref{ConvergenceZ_hToZ} and \eqref{ConvergenceZ_hToZ-t0}, we can pass to the limit in the right hand side of \eqref{MultiplierDiscrete}, which converges to the right hand side of \eqref{multid}, leading to: 
$$
	\lim_{ h \to 0} \left( \frac{1}{2} \int_0^T |(\partial_h^- \partial_t z_h)_{N+1}|^2\, dt
	+ 
	\frac{h^2}{4} \int_0^T \int_{[0,1)} |\partial_h ^+ \partial_{tt} z_h|^2 \, dt
	 \right) 
	 =
	 \frac{1}{2} \int_0^T  |\partial_{xt} z(\cdot,1)|^2 \, dt.
$$
This last fact, together with the weak convergence \eqref{WeakConvergence}, implies that
\begin{equation}
	\label{FhConvZh}
	((\partial_h^- \partial_t z_h)_{N+1}, h \partial_x \eh( \partial_{tt} z_h)) \underset{h \to 0} \longrightarrow (\partial_{xt} z(\cdot, 1), 0) \ \hbox{ strongly in } L^2(0,T) \times L^2((0,T)\times(0,1)).
\end{equation}

\noindent {\bf Step 2 Convergence with source term and potential}\\
So far, we did not assume anything on the potential $p$ and on the convergence of the discrete potentials $p_h$ to $p$, since they did not appear in the study of $z$ and $z_h$ solutions of \eqref{CWE1z} and~\eqref{SDWE1z}.

Since $p \in L^\infty_{\leq m}(0,1)$, it is very easy to construct a sequence $p_h \in L^\infty_{h, \leq m} (0,1)$ such that $\eh^0 (p_h)$ strongly converge to $p$ in $L^2(0,1)$. Taking such sequence $p_h$,
we set $y_h[p_h] = z_h + v_h[p_h]$ where $v_h[p_h]$ is the solution of
\begin{equation*}
	\left\{\begin{array}{ll}
		\partial_{tt}v_{j,h} - \left(\Delta_h v_{h}\right)_j + p_{j,h} v_{j,h}= g_{j,h} - p_{j,h} z_{j,h}, \quad & t\in (0,T), \ j\in \llbracket 1,N \rrbracket,\\
		v_{0,h}(t)= v_{N+1,h}(t)=0,& t\in [0,T],\\
		v_{j,h}(0)= 0, \quad\partial_t v_{j,h}(0)= 0, &  j\in \llbracket 1,N \rrbracket.
	\end{array}\right.
\end{equation*}
Due to the convergence hypothesis in Assumption \ref{AssConvergence}, we have the convergence of $\eh^0(g_h) $ towards $g $ in $W^{1,1}(0,T; L^2(0,1))$ as $h \to 0$.

From \eqref{ConvergenceZ_hToZ}, $\eh^0 (z_h)$ and $\eh^0(\partial_t z_h)$ strongly converge in $L^2(0,T; L^2(0,1))$ towards $z$ and $\partial_t z$, respectively. Besides, $\eh(z_h)$ and $\eh(\partial_t z_h)$ respectively converge to $z$ and $\partial_t z$ in $L^2(0,T; H^1(0,1))$. Since $z_{0,h}(t) = g_h^0(t)$ and $\partial_t z_{0,h}(t) = \partial_t g^0_h(t)$ are bounded ($ H^2 (0,T) \subset C^1([0,T])$) and strongly converge in $C^0([0,T])$, respectively, toward $z(t,0) = g^0(t)$ and $\partial_t z(t,0) = \partial_t g^0(t)$,  $\eh(z_h)$ and $\eh(\partial_t z_h)$ respectively converge to $z$ and $\partial_t z$ in $L^2(0,T; L^\infty(0,1))$. Of course, this implies the $L^2(0,T; L^\infty(0,1))$ boundedness of the sequences $\eh^0 (z_h)$ and $\eh^0(\partial_t z_h)$. Since they converge strongly in $L^2(0,T; L^2(0,1))$, we also have the strong $L^2(0,T;L^4(0,1))$-convergences of $(\eh^0(z_h), \eh^0(\partial_t z_h))$ toward $(z, \partial_t z)$.

Similarly, $\eh^0(p_h)$ strongly converges to $p$ in $L^2(0,1)$ and is bounded in $L^\infty(0,1)$. Therefore, $\eh^0(p_h)$ strongly converges to $p$ in $L^4(0,1)$.

Since $\eh^0 (a_h b_h) = \eh^0(a_h) \eh^0 (b_h)$, we thus obtain that $\eh^0 (p_h z_h)$ and $\eh^0(p_h \partial_t z_h)$, respectively, strongly converge to $pz$ and $p \partial_t z$ in $L^2(0,T; L^2(0,1))$.

Therefore, 
$$
	\eh^0(g_h - p_h z_h) \underset{h\to 0}{\longrightarrow} g - pz \quad \hbox{in }W^{1,1}(0,T; L^2(0,1)).
$$

Thus, classical results yield the convergence of $v_h[p_h]$ toward $v[p]$, solution of 
\begin{equation*}
	\left\{\begin{array}{ll}
		\partial_{tt}v -\Delta v + p v = g - p z, \quad & (t,x)\in (0,T)\times (0,1),\\
		v(t,0)=v(t,1)=0,& t\in (0,T),\\
		v(0,\cdot)= 0, \quad\partial_t v(0,\cdot)= 0. &
	\end{array}\right.
\end{equation*}
Therefore, we obtain
\begin{equation}
	\label{ConvergenceV_htoV}
	\left\{
		\begin{array}{l}
	\partial_{t} \eh( v_h[p_h] ) \underset{h \to 0}{\longrightarrow} \partial_{t} v[p] \quad \hbox{strongly in } L^2((0,T); H^1_0 (0,1))\cap H^1 ((0,T);L^2(0,1)),
	\\
	\eh ( v_h[p_h] ) \underset{h \to 0}{\longrightarrow} v[p]  \quad \hbox{strongly in } L^2((0,T); H^1_0(0,1))\cap H^2 ((0,T);L^2(0,1))
		\end{array}
	\right.
\end{equation}
and, for all $t_0 \in [0,T]$,
$$
	(\eh^0 (\partial_h \partial_t v_h[p_h] )(t_0), \eh^0 (\partial_{tt} v_h[p_h] )(t_0))  \underset{h \to 0}{\longrightarrow} (\partial_{xt} v[p] (t_0), \partial_{tt} v[p](t_0))  \quad \hbox{strongly in } (L^2(0,1))^2.
$$
Of course, as for $z_h$, using the discrete multiplier identity satisfied by $\partial_t v$ (see \eqref{MultiplierDiscrete}) and the above convergences, we easily get
\begin{multline}
	\label{FhConvVh}
	((\partial_h^- \partial_t v_h[p_h])_{N+1},  h \partial_x \eh( \partial_{tt} v_h[p_h])) \underset{h \to 0} \longrightarrow (\partial_{xt} v[p](\cdot, 1), 0) 
	\\ \hbox{ strongly in } L^2(0,T) \times L^2((0,T)\times(0,1)).
\end{multline}

Now, using \eqref{FhConvZh} and \eqref{FhConvVh}, the solution $y_h[p_h]$ of \eqref{SDWE1} converges to the solution $y[p]$ of \eqref{CWE1} in the following sense:
$$
	((\partial_h^- \partial_t y_h[p_h])_{N+1}, h \partial_x \eh( \partial_{tt} y_h[p_h])) \underset{h \to 0} \longrightarrow (\partial_{xt} y[p](\cdot, 1), 0) \ \hbox{ strongly in } L^2(0,T) \times L^2((0,T)\times (0,1)),
$$
which is precisely \eqref{ConsistanceFh}.

Besides, using \eqref{ConvergenceZ_hToZ} and \eqref{ConvergenceV_htoV}, we have
$$
	\eh ( y_h[p_h] ) \underset{h \to 0}\longrightarrow  y[p] \hbox{ in } H^1(0,T; H^1(0,1)), 
$$
which of course implies the bound \eqref{BoundsOnDtYh}.
This concludes the proof of Theorem \ref{ThmConsistence}.
\end{proof}

We are now in position to prove Theorem \ref{ThmCV}.

\begin{proof}[Proof of Theorem \ref{ThmCV}]
%
Let $p \in L^\infty_{\leq m}(0,1)$ and let $q_h \in L^\infty_{h,\leq m}(0,1)$ be such that \eqref{ConvergenceObs} holds. Denote by $p_h$ the potentials given by Theorem \ref{ThmConsistence}.
Then we have 
$$
	\Theta_h(p_h) - \Theta_h(q_h) \underset{h \to 0}\longrightarrow (0,0) \quad \hbox{ strongly in } L^2(0,T) \times L^2((0,T)\times (0,1)).
$$
But according to \eqref{BoundsOnDtYh} and the positivity Assumption \ref{AssPositivity}, 
we can apply Theorem~\ref{TCWE2}: for some $C>0$ independent of $h>0$, estimate \eqref{UniformStability-4} holds. Therefore, $\eh^0 (p_h) -\eh^0(q_h)$ strongly converges to zero in $L^2(0,1)$. Using \eqref{ConsistanceFh}, we deduce that $\eh^0( q_h ) $ strongly converges to $p$ in $L^2(0,1)$.
\end{proof}

\section{Further comments}\label{SecConclusion}

\hspace{4ex}$\bullet$ {\bf Other convergence results.} Note that our convergence results require the convergence of $h \partial_h^+ \partial_{tt} y_h[p_h]$ to zero in $L^2(0,T;L^2(0,1))$. This term is here to handle spurious high-frequency waves generated by the space semi-discretization - see e.g. \cite{Tref} - which are by now well-known to be responsible for the lack of uniform observability of waves \cite{Zua05Survey}. Of course, other ways of removing these high-frequency waves can be implemented, an easy one being to impose some smoothness and filtering conditions on the data - see Remark \ref{RemFiltering}. Note however that these conditions seem to be more difficult to implement in practice.

$\bullet$ {\bf Time discretization.} Here we focused on the space semi-discretization of the wave equation for simplicity. Indeed, the fully discrete wave equation in which the time-derivative has been approximated by the centered difference approximation could be handled the same way, since time and space are completely decoupled then. This will of course introduce a Tychonoff regularization term within the Carleman estimates of the same order but depending not only on the space discretization parameter, but also on the time semi-discretization parameter. This again is completely compatible with the known results on the observability of discrete waves - see \cite{je3,ErvZuaBook}. 

$\bullet$ {\bf Other space discretizations.} Here, we have chosen a very simple space discretization process corresponding to the finite-differences approach. Other space discretizations should be studied, but regarding the literature in what concerns discrete observability estimates for the waves (see e.g. \cite{Zua05Survey,ErvZuaCime}), we expect the Tychonoff regularization term to be needed within the discrete Carleman estimates in the case of finite-elements methods. However, for mixed finite elements methods (see \cite{CasMic,CasMicMunch,je2}), we expect better behavior than here and this Tychonoff regularization term may be not needed anymore. This should be studied carefully. 

$\bullet$ {\bf Higher dimensions and more sophisticated wave models.} Of course, an interesting question would be to develop these discrete Carleman estimates in higher dimensions (as it was done in \cite{BoyerHubertLeRousseau2}). It is usually admitted that Carleman estimates  ``do not see" the dimension of the space. This is indeed true in the continuous case, but in the discrete case,  the integrations by parts are much more intricate. This is currently under investigation. Regarding more generic hyperbolic models, one could also mention \cite{ImYamCOCV05}, \cite{BMO07} or \cite{BaudouinCrepeauValein11} giving stability of inverse problem from global Carleman estimates respectively for the Lam\'e system, a discontinuous wave equation or in a network of 1-d strings.

$\bullet$ {\bf Semilinear wave equations.} One of the standard applications of Carleman estimates is to prove controllability of semilinear wave equations - see \cite{DuyckaertsZhangZuazua,FuYongZhang}. We expect that these discrete Carleman estimates could be of some use to prove the convergence of discrete controls for semilinear wave equations and to improve the results already obtained for globally Lipschitz nonlinearities using bi-grids methods  in \cite{Zuazua06ICM}.  

$\bullet$ {\bf How to compute a discrete sequence $p_h$ such that $\Theta_h(p_h)$ converges to $\Theta_0(p)$ ?} This is certainly one of the most challenging issues concerning this kind of inverse problems, since the map $\Theta_h$ is highly nonlinear. Of course, a natural idea is to introduce
$$
	J_h(p_h) = \norm{\Theta_h(p_h) -\Theta_0(p)}_{L^2(0,T) \times L^2(0,T;L^2(0,1))}^2
$$
and to minimize it on the set $L^\infty_{h, \leq m} (0,1)$. But this can be very hard since $J_h$ may have several local minima. 
Another approach will be presented in the work \cite{BaudouinDeBuhanErvedoza} based on Carleman estimates and stability results inspired from \cite{ImYamCom01,Baudouin01}.
\\

\noindent{\bf Acknowledgements.} The authors acknowledge J\'er\^ome Le Rousseau, Franck Boyer, Jean-Pierre Puel, Masahiro Yamamoto and Fr\'ed\'eric De Gournay for interesting discussions related to that work.\\

\bibliographystyle{plain}

\end{document}